\documentclass[10pt,a4paper]{article}

\usepackage{lineno,hyperref}
\modulolinenumbers[5]
\usepackage[utf8]{inputenc}
\usepackage{amsthm,amsmath,enumerate,epic}
\usepackage[mathscr]{euscript}
\usepackage{amssymb}
\usepackage{graphicx}
\usepackage{caption}
\usepackage{subcaption}
\usepackage{parskip}
\hypersetup{colorlinks=true,linkcolor=blue,citecolor=blue}
\newtheorem{thm}{\bf Theorem}[subsection]
\newtheorem{example}{Example}[subsection]
\newtheorem{prop}{Proposition}[subsection]
\newtheorem{remark}{\bf Remark}[thm]
\newtheorem{defi}{\bf Definition}[subsection]
\newtheorem{pro}{\bf Proposition}[subsection]

\newtheorem{theorem}{Theorem}[subsection]
\newtheorem{lemma}{\bf Lemma}[subsection]

\newenvironment{acknowledgements}{%
\par\noindent{\textbf{Acknowledgements:}}\begin{rm}}{\end{rm}\newline}
\numberwithin{equation}{subsection}

\title{On the homogeneity problem of the Kazhdan-Lusztig ideals}
\date{}
\author {Shyamashree Upadhyay\\
Department of Mathematics\\ Indian Institute of Technology Guwahati\\Assam-781039, INDIA\\email: shyamashree@iitg.ac.in\\and\\Adhip Ganguly\\Department of Mathematics\\ Indian Institute of Technology Guwahati\\Assam-781039, INDIA\\email: g.adhip@iitg.ac.in}
\begin{document}
\maketitle



\begin{abstract}
In this paper, our purpose is to find some techniques to identify inhomogeneous Kazhdan-Lusztig ideals. A set of generators of a Kazhdan-Lusztig ideal corresponds to some minors in some southwest submatrices inside some ambient matrix. We have defined a notion called a ``path'' inside a square matrix as a tool to work with determinants of the above mentioned minors. Using paths, we derive necessary and sufficient conditions for when such minors are inhomogeneous, singular, non-singular, and when there is divisibility between two terms of two such distinct minors. As a conclusion, we first provide an algorithm through which we can identify some sufficient conditions for a Kazhdan-Lusztig ideal to be inhomogeneous (by utilizing all the above mentioned results). Secondly, we provide certain sufficient conditions for a Kazhdan-Lusztig ideal to be standard homogeneous, and then we also prove that under certain assumptions, these conditions are also necessary for a Kazhdan-Lusztig ideal to be standard homogeneous.
\end{abstract}
\noindent \textbf{Mathematics subject classification (2020):}  14N10; 05E14; 05E40; 13C70.\\
\noindent \textbf{Keywords:} 
Inhomogeneous. Kazhdan-Lusztig ideal. Standard homogeneous. Path. Determinant. Essential minor. Polynomial mutation. $k$ band.\\



\tableofcontents
\section{Introduction}\label{s.Introduction}
Let $\mathbb{C}$ denote the set of complex numbers, $GL_n(\mathbb{C})$ be the set of all invertible $n\times n$ matrices over $\mathbb{C}$, we often will refer to this as $GL_n$ as the underlying field will be $\mathbb{C}$ throughout the paper unless mentioned otherwise. Let $\mathbb{N}$ denote the set of natural numbers and for any $n\in \mathbb{N}$, let $[n] := \{1.2,\dots ,n\}$. Let $S_n$ denote the symmetric group on $n$ letters.\par 
Given any natural number $n$, and any $n\times n$ matrix $A=(a_{ij})$, where the $a_{ij}$ are either entries from some field or some indeterminates, the ideal obtained by taking any number of $r\times r$ minors (for some positive integer $r\leq n$) is called a determinantal ideal corresponding to the matrix $A$. Now, usually if the $a_{ij}$ are unspecialized variables, then it is clear that the minors are homogeneous, and hence the determinantal ideal is also homogeneous. However, this may not be the case always. However, often in some standard grading, the determinantal ideals turn out to be homogeneous. So, the main problem remains, that is, when is a determinantal ideal homogeneous in a weighted or multi graded context. For example, in the paper \cite{Conca1}, they consider ideals generated by minors of a matrix of linear forms which is row-graded or column-graded (that is, the entries have many different multidegrees via $\mathbb{Z}^n$-grading) and show that the ideal is a Cartwright-Sturmfels ideal.\par
In the paper \cite{Eagon}, Eagon and Northcott consider determinantal ideals generated by the maximal minors of a matrix with entries in a commutative ring. These ideals are known as the Eagon-Northcott ideals. These ideals are associated with a specific free complex, now known as  the Eagon-Northcott complex, which provides a free resolution of the ideal under certain depth conditions. Later, in \cite{Rim}, Buchsbaum and Rim extended these constructions to more general module homomorphisms, further explaining the connection between determinantal ideals and the homological properties of the underlying maps.\par
Recent investigation (see \cite{Bruns},\cite{Conca2},\cite{Miller}) has focused on identifying precise criteria under which determinantal ideals remain homogeneous with respect to arbitrary gradings (as well as on explaining the consequences of their Rees algebras, initial ideals, and associated projective varieties). They reveal interaction between combinatorial data and graded algebraic geometry.\par 
The Kazhdan-Lusztig ideal is a type of determinantal ideal, and is algebraically defined as the ideal of the Kazhdan-Lusztig variety, denoted by $N_{v,w}$, which is the intersection of a Schubert variety and an opposite Schubert cell (see \S 3 of \cite{woo2023schubert}). One can use this definition to interpret  Kazhdan-Lusztig Ideals simply in terms of determinantal ideals as follows:
\par 
Given $w\in S_n$, the \textit{rank matrix of $w$}, denoted $R_w$,is the matrix 
\[
R_w :=[r_{p,q}^{(w)}]_{p,q=1}^n
\]
where the entry $r_{p,q}^{(w)}= \# \{\ k | \ k \leq q, \ w(k) \leq p\}$.\par 

Analogous to $R_w$, set $\Tilde{R}_w$ as the matrix $\Tilde{R}_w$ := $[\Tilde{r}_{p,q}^{(w)}]_{p,q=1}^n$ where the entry
\begin{equation}
	\Tilde{r}_{p,q}^{(w)} = \#\{ k \ | \ k \leq q, w(k) \geq p \}. \notag
\end{equation}
Let $Mat_{n\times n}$ be the space of all $n \times n$ complex matrices. Consider the ring $\mathbb{C}$[\textbf{z}] where \textbf{z}=$\{z_{ij}\}_{i,j=1}^n$ are the functions on the entries of a generic matrix Z of $Mat_{n\times n}$. Here,\\
$z_{ij} = \text{the entry in the $i$-th row from the $bottom$, and the $j$-th column to the left.}$
\par
Consider the matrix $Z^{(v)}$ where $z_{n-v(i)+1,i}$ = 1, and $z_{n-v(i)+1,s}$ = 0, $z_{t,i}$ = 0 for $s>i$ and $t>n-v(i)+1$. Let \textbf{z}$^{(v)} \subset \textbf{z} $ be the unspecialized variables.\\
Furthermore, let $Z_{st}^{(v)}$ be the southwest $(n-s+1) \times t$ submatrix of $Z^{(v)}$.\\
The \textit{Kazhdan-Lusztig ideal} (KL-ideal) is the ideal $I_{v,w} \subset \mathbb{C}$[\textbf{z$^{(v)}$}] generated by all $\Tilde{r}_{st}^{(w)} + 1$ minors of $Z_{st}^{(v)}$
where $1 \leq s,t\leq n$ and $\Tilde{r}_{st}^{(w)}$=$\#\{k \ | \ k\leq t,w(k)\geq s\}$.\\
One can check that $N_{v,w}$ is set-theoretically cut out by $I_{v.w}$, i.e, $p\in N_{v,w}$ iff $p$ is a zero of the generators of $I_{v,w}$.\par 

So, one can readily see that the \textit{Kazhdan-Lusztig ideal} $I_{v,w}$ can also be defined as the ideal of $\mathbb{C}$[\textbf{z$^{(v)}$}] generated by all $\Tilde{r}_{n-s+1,t}^{(w)}+1$ minors of $Z_{st}^{(v)}$, where $Z_{st}^{(v)}$ denotes the southwest $s\times t$ submatrix of $Z^{(v)}$, and $\Tilde{r}_{n-s+1,t}^{(w)}=\#\{k \ | \ k\leq t,w(k)\geq n-s+1\}$, where $1\leq s,t\leq n$.\\
In this paper, we work with the previous definition mostly except maybe at one or two places where it will be clear from the context.\par 
$I_{v,w}$ may contain, among its generators, inhomogeneous polynomials. We say $I_{v,w}$ is \textit{standard homogeneous} if there exist homogeneous polynomials that generate $I_{v,w}$. In this case, we also say $N_{v,w}$ is standard homogeneous.\par
In this paper, we first provide an algorithm, through which one can systematically reduce the initial set $S_n\times S_n$ to a much smaller set on which one can then identify the inhomogeneous KL-ideals (via some additional procedure). Secondly, we provide certain necessary and sufficient conditions for identifying homogeneous KL-ideals, under certain assumption on $v,w$. In the paper \cite{neye2023grobner}, Emmanuel Neye has shown that if $v$ is 321- avoiding or $w$ is 132-avoiding, then $I_{v,w}$ is standard homogeneous (see \cite[Proposition 6.4]{neye2023grobner}). The same Proposition also gives certain conditions for $I_{v,w}$ to be ``standard homogeneous'' in case $v$ is non 321-avoiding and $w$ is not 132-avoiding.
For more analysis and different results regarding the topic, one may check \cite{woo2012grobner} and \cite{woo2008governing}.  In this paper of ours, we provide an insight to the problem of analyzing the homogeneity of those KL-ideals $I_{v,w}$, for which $v$ and $w$ do not necessarily fall into Neye's set\footnote{We call the set of all $(v,w)\in S_n\times S_n$ which satisfy the conditions (a) or (b) given in \cite[Proposition 6.4]{neye2023grobner}\label{f.Neyeset} as Neye's set.}. We do not use any advanced techniques in this paper, all techniques used here are elementary in nature.\par
\noindent\textbf{Organization of the paper:} In \S \ref{s.stateproblem}, we define objects called the ``defining minors'' of the KL ideal (which are the defining generators of the KL ideal), and then we state the problem we consider in this paper. In \S \ref{s.discw}, we show that in fact, for the purpose of our problem, we can discard the permutation $w=n\ n-1\ \cdots\ 2\ 1$. In this section, we also show that for the purpose of our problem, we can also discard those pairs $(v,w)\in S_n\times S_n$ such that $\Tilde{R}_v\nleq\Tilde{R}_w$ (see Definition \ref{d.AnleqB} for the meaning of this notation). In \S \ref{s.redminors}, we give a formula of an arbitrary element in the rank matrix $\tilde{R_w}$ and also show a systematic way to remove redundant defining minors. In \S \ref{s.resonhom}, we provide a necessary condition for homogeneity of any ideal in a polynomial ring over $\mathbb{C}$. In \S \ref{s.paths}, we first introduce the notion of paths, its relation with the determinant, and we use the definition of a path to derive a necessary and sufficient condition for a minor to have an inhomogeneous determinant. Secondly, we have also given conditions for when a minor inside a submatrix $Z_{st}^{(v)}$ of $Z^{(v)}$ is singular. Thirdly, we also give a necessary and sufficient condition for when one term of the determinant of a minor divides a term of the determinant of another such minor. Moreover, we also provide a theorem which gives a necessary and sufficient condition for a minor in $Z^{(v)}$ to have inhomogeneous determinant. In \S \ref{s.whichwtodiscard}, we have used the contents of the previous sections to give an algorithm to determine when a KL ideal $I_{v,w}$ is inhomogeneous. In the next section (\S \ref{s.neye}), we have given an example on which the algorithm given in \S \ref{s.whichwtodiscard} has been executed. Thereafter, in \S \ref{s.poly-mutation}, we introduce a procedure called ``polynomial mutation'', which helps us in the last section (\S \ref{s.iff}) to obtain certain sufficient conditions for a KL-ideal $I_{v,w}$ to be standard homogeneous. Moreover, we also prove that under some assumption on $v$ and $w$, these conditions are also necessary for a KL-ideal to be standard homogeneous. The main results of the paper are Proposition \ref{p.statement2andconverse} and Theorem \ref{t.main-work}. Finally, in \S \ref{ss.related-to-neye}, we relate the results obtained by us in this paper with those KL-ideals $I_{v,w}$, for which the $v,w$ are in the complement of Neye's set (see footnote \ref{f.Neyeset} below).
\section{Stating our problem}\label{s.stateproblem}
The Kazhdan-Lusztig (KL)-ideal is as defined in \S \ref{s.Introduction}.
\subsection{The ``defining minors" of the KL ideal.}\label{ss.defminors}
In the introduction, we have defined the Kazhdan-Lusztig ideal and notations $Z^{(v)}$ and $Z_{st}^{(v)}$. We follow those notations from here on.\\
All $p\times p$ minors of $Z^{(v)}$ are of the form:
\[
\begin{bmatrix}
	z_{i_pj_1} & z_{i_pj_2} & \dots & z_{i_pj_p}\\ 
	z_{i_{p-1}j_1} & z_{i_{p-1}j_2} & \dots & z_{i_{p-1}j_p}\\
	\vdots & \vdots &\vdots & \vdots\\
	z_{i_1j_1} & z_{i_1j_2} & \dots & z_{i_1j_p}
\end{bmatrix}
\]
where 
\[
z_{i_kj_q}=\begin{cases}
1 & \text{if}\quad i_k=n-v(j_q)+1,\\
0 & \text{if}\quad i_k > n-v(j_q)+1\\
0 & \text{if}\quad i_k=n-v(j)+1\ \mbox{for}\ \mbox{some}\ j < j_q\\
z_{i_kj_q} & \text{otherwise.}	
\end{cases}	
\]
where $i_1<i_2<\dots<i_p$ , $j_1<j_2<\dots<j_p$ and  $1\leq p\leq n-1$.\par 
Notice that $p\leq n-1 $ follows from the fact that $\tilde{r}_{1,n}^{(w)}=n$. In fact, $p\leq n-2$ for all $w\in S_n$ that send n to 1.\par 
Also, if $p=\tilde{r}_{s,t}^{(w)}+1$ for some $s,t \in\{1,2,\dots,n\}$, then $i_p\leq n-s+1, j_p\leq t$ and $p\leq min\{n-s+1,t\}$.\par 
We  call these $p\times p$ minors as the \textbf{defining minors of the KL ideal}.
\begin{example}\label{e.defminors}
Let $v=2314 $ and $w=4213$ be elements of $S_4$.\\
	Then, \[
	\tilde{R}_w=	
	\begin{bmatrix}
		1&2&3&4\\
		1&2&2&3\\
		1&1&1&2\\
		1&1&1&1
	\end{bmatrix}
	\quad
	Z^{(v)}=
	\begin{bmatrix}
	0&0&1&0\\
	1&0&0&0\\
	z_{21}&1&0&0\\
	z_{12}&z_{12}&z_{13}&1
	\end{bmatrix}	
	\]
and, the defining minors are :
\[
\begin{bmatrix}
	z_{21}&1\\
	z_{11}&z_{12}
\end{bmatrix},
\quad
\begin{bmatrix}
	z_{21}&0\\
	z_{11}&z_{13}
\end{bmatrix},
\quad
\begin{bmatrix}
	1&0\\
	z_{11}&z_{13}
\end{bmatrix} \textit{ and }
\begin{bmatrix}
	1&0&0\\
	z_{21}&1&0\\
	z{11}&z_{12}&z_{13}
\end{bmatrix}.
\]
Notice that $\nexists$ any defining minors for the southwest $s\times 1$ and $1\times t$ submatrices $Z_{n-s+1,1}^{(v)}$ and $Z_{n,t}^{(v)}$ respectively for any $s,t \in \{1,2,\dots,n\}$ as $\tilde{r}_{s1}^{(w)}=1$ and $\tilde{r}_{1t}^{(w)}\geq 1$ $ \forall \ 1\leq s,t \leq 4$ and hence $\tilde{r}_{s1}^{(w)}$ $(\textnormal{ or } \tilde{r}_{1t}^{(w)})+1 \nleq min\{s,1\}(\textnormal{ or } min\{1,t\})=1$.
\end{example}

\subsection{The Problem}\label{ss.problem}
Let $I_{v,w}$ be as in \S \ref{s.Introduction}. The problem is to classify all $(v,w) \in S_n \times S_n$ such that $I_{v,w}$ is standard homogeneous (the notion of standard homogeneity is as given in \S \ref{s.Introduction}). We have partially solved this problem.\par
In this paper, we first provide an algorithm through which we can identify some sufficient conditions for a Kazhdan-Lusztig ideal to be inhomogeneous. Secondly, we provide certain sufficient conditions for a Kazhdan-Lusztig ideal to be standard homogeneous, and then we also prove that under certain assumptions, these conditions are also necessary for a Kazhdan-Lusztig ideal to be standard homogeneous.
\section{Which $(v,w)$'s to discard initially?}\label{s.discw}
In this section, we provide two results, which help us to discard certain pairs $(v,w)\in S_n\times S_n$ while proceeding with our search for inhomogeneous KL-ideals. These two results are well known, but the proofs we provide are much simple.
\subsection{Two results}\label{ss.aresult}
\begin{lemma}\label{l.nodefminors}
	There are no defining minors for $I_{v,w}$ ,i.e, $I_{v,w}=\emptyset \Longleftrightarrow w=n \ n-1 \ n-2 \ \dots \ 1$. 
\end{lemma}
\begin{proof}
$\implies$: Let $p,q\in\{1,\ldots,n\}$ be arbitrary. First, suppose that $q=1$. Starting with a matrix $Z^{(v)}$, first look at the $p\times 1$ sub-matrix. If $\tilde{r}_{p1}^{(w)}=0$ for any $p$, then $\exists$ a minor of size 1 and so, the set of generators of $I_{v,w}$ is non empty. So, we must have $\tilde{r}_{p1}^{(w)}=1$ $\forall$ $1\leq p\leq n$ which holds iff $w(1)=n$. Hence we have $w(1)=n$.\\
Now, clearly for $q\geq 2$, if $\tilde{r}_{pq}^{(w)}=q$, then there are no minors for those $(n-p+1) \times q$ sub-matrices. Notice that $\tilde{r}_{pq}^{(w)}=q \iff p\leq min\{w(i)|i\leq q\}$. Hence, $\tilde{r}_{pq}^{(w)}\neq q$ $\forall$ $p> min\{w(i)|i\leq q\}$. Also, as $q\geq 2$, $min\{w(i)|i\leq q\}<n$ $\implies \exists p$ such that $\tilde{r}_{pq}^{(w)}<q$. For these $Z_{pq}^{(v)}$ (that is, for those $Z_{pq}^{(v)}$, for which $q\geq 2$ and $p$ is such that $\tilde{r}_{pq}^{(w)}<q$), to guarantee that there are no minors that can be taken, we must have $n-p+1< \tilde{r}_{pq}^{(w)}+1$. Using $\tilde{r}_{pq}^{(w)}<q$, we have $ n-p+1<q$, which is true if and only if $n-q+1<p$. This says that for all $p > n-q+1$, we have that although, $\tilde{r}_{pq}^{(w)}<q$, but $\tilde{r}_{pq}^{(w)}+1 > n-p+1$, i.e, no minors are contributed from these $(n-p+1)\times q$ sub-matrices of $Z^{(v)}$.\par 
Now, we focus our attention to $p\in \{1,2,\dots , (n-q+1)\}$ (that is, $p\leq n-q+1$) such that $\tilde{r}_{pq}^{(w)}<q$. For all such $p$, we must have again that, $\tilde{r}_{pq}^{(w)}+1 > n-p+1$. Putting $p=n-q+1$, we get: $\tilde{r}_{n-q+1,q}^{(w)}+1>q$ $\iff \tilde{r}_{n-q+1,q}^{(w)}>q-1$. Now it follows from some basic properties of the matrix $\Tilde{R}_w$ that we must have $\tilde{r}_{n-q+1,q}^{(w)}=q$, which in turn implies that $min\{w(i)|i\leq q\}\geq n-q+1$ (by definition of the $\Tilde{R}_w$ matrix).\par 
On the other hand, we have $min\{w(i)|i\leq q\}\leq n-q+1$ simply by counting. Hence together, we have : $min\{w(i)|i\leq q\}=n-q+1$.\par 
Now, this is true for all $q\geq 2$. Taking $q=2$ in the above, we get $min\{w(i)|i\leq 2\}=n-1$, and as we already have $w(1)=n$, it must be that $w(2)=n-1$. Iteratively, in a similar manner, we see that for each $1\leq i\leq n$, $w(i)=n-i+1$, i.e, $w=n \ n-1\  n-2 \ \dots 1$. \\ \\
$\impliedby$: Let $w_0=n\ n-1\ n-2\ \dots\ 1$. For $w=w_0$, we have:
$$\tilde{r}_{st}^{(w)}=\left\{\begin{array}{cc}
  t & \forall s\in\{1,\ldots,n-t+1\}\\
  \alpha & \mbox{for}\ s=n-\alpha+1\ \mbox{where}\ 1\leq\alpha\leq t-1
                              \end{array}
\right.$$
Thus, clearly for $s\in\{1,\ldots,n-t+1\}$, we have $\tilde{r}_{st}^{(w)}+1=t+1>min\{n-s+1,t\}$. And for $s\notin\{1,\ldots,n-t+1\}$, it must be that $s=n-\alpha+1$ for some $\alpha$ such that $1\leq\alpha\leq t-1$. For $s$ of the later type, $\tilde{r}_{st}^{(w)}+1=\alpha+1$, that is, we are supposed to take an $(\alpha+1)$ minor in a southwest $(n-s+1)\times t=\alpha\times t$ matrix, which is absurd. That is, for any fixed $t$, there does not exist any $s$ such that the minors of $Z_{st}^{(v)}$ make sense. Hence $I_{v,w}=\emptyset$.
\end{proof}
Lemma \ref{l.nodefminors} above tells us that we all discard all pairs $(v,w)$ for which $w=n\ n-1\ n-2\ \dots\ 1$. Next, we will prove a theorem, which will tell us that we can also discard those pairs $(v,w)$ for which $\Tilde{R}_v\nleq\Tilde{R}_w$.
\begin{defi}\label{d.AnleqB}
We say for any two matrices $A=[a_{ij}]$ and $B= [b_{ij}]$, $A\leq B$ iff $a_{ij}\leq b_{ij}$ for each $i,j$ . Else, we say that $A\nleq B$.
\end{defi}

\begin{theorem}\label{t.wholeRing}
	The Kazhdan-Lusztig Ideal $I_{v,w}$ is whole ring $\mathbb{C}[z_{ij}]$ if and only if $\tilde{R}_v \nleq \tilde{R}_w$.
\end{theorem}
\begin{proof}
	Suppose $I_{v,w}$ contains 1. Then, by \hyperref[o.o1]{Observation 1} (which appears in a later section, namely, \S \ref{s.paths} as an independent result), there are some $s,t \in \{1,\dots, n\}$ such that in the southwest $(n-s+1)\times t$ matrix $Z_{st}^{(v)}$, at least one of its  $\tilde{r}_{s,t}^{(w)}+1$ minors has determinant $1$ or $-1$. This implies that there exist at least $\tilde{r}_{st}^{(w)}+1$ many rows in the southwest $(n-s+1)\times t$ matrix $Z_{st}^{(v)}$, such that they have $1$ in them. Now, a row has a $1$ in it means that there exists some $i$-th column such that the corresponding $v(i)$ value equals the row number from the top. Here, any such $i$ must satisfy $i\leq t$ because we are in the southwest $(n-s+1)\times t$ matrix. As there are totally $n$ rows in $Z^{(v)}$, and $n-(n-s+1)=s-1$, so it must be that there are $(s-1)$ rows above the matrix $Z_{st}^{(v)}$. This will mean that for the $i$'s as mentioned above, the $v(i)$-value is a row number strictly greater than $(s-1)$. That is, $v(i)\geq s$.\par
	So the fact that there exist at least $\tilde{r}_{st}^{(w)}+1$ rows with $1$ in them (in the south-west submatrix $Z_{st}^{(v)}$) $\implies$ there exist at least $\tilde{r}_{st}^{(w)}+1$ many $i$'s such that $i\leq t$ and $v(i)\geq s$. That is, $\#\{i\leq t|v(i)\geq s\}\geq\tilde{r}_{st}^{(w)}+1$, which is the same as saying that $\tilde{r}_{st}^{(v)}\geq\tilde{r}_{st}^{(w)}+1$. Hence, for this particular $s,t$, we have $\tilde{r}_{st}^{(v)}>\tilde{r}_{st}^{(w)}$. So we cannot have $\Tilde{R}_v\leq\Tilde{R}_w$. That is, we must have $\Tilde{R}_v\nleq\Tilde{R}_w$.\par
	
	On the other hand, if $\tilde{R}_v \nleq \tilde{R}_w$ then, $\exists$ some $s,t$ such that $\tilde{r}_{s,t}^{(v)} > \tilde{r}_{s,t}^{(w)}$. This implies $\tilde{r}_{s,t}^{(v)} \geq \tilde{r}_{s,t}^{(w)}+1$. Notice that the previous line also says that $\tilde{r}_{s,t}^{(w)}+1 \leq min\{n-s+1,t\}$ and thus $\tilde{r}_{s,t}^{(w)}+1$ minor in  $Z_{st}^{(v)}$ makes sense. Let $\tilde{r}_{st}^{(w)}+1=b$. Then because $\tilde{r}_{st}^{(v)}\geq\tilde{r}_{st}^{(w)}+1=b$, it follows that $\#\{i\leq t|v(i)\geq s\}=\tilde{r}_{st}^{(v)}\geq b$. This means that there exists at least $b$ many $i$'s (say, $i_1,\ldots,i_b$) such that $i\leq t$ and $v(i)\geq s$. This implies that in the southwest $(n-s+1)\times t$ matrix, there exist at least $b$ rows that have $1$ in them, because $v(i)\geq s$. For all those rows (say, $\alpha_1,\ldots,\alpha_b$), it must have been that $v(i_k)=\alpha_k$ for $1\leq k\leq b$, and also each $i_k\leq t$.\par 
	
	Take the rows $\alpha_1,\ldots,\alpha_b$ and the columns $i_1,\ldots,i_b$. They form a minor inside $Z_{st}^{(v)}$, which has a $1$ in each of its rows. As there is a $1$ in each of the rows of such a minor, its determinant will have $1$ or $-1$ as one of its terms. It then follows from \hyperref[o.o1]{Observation 1} that such a determinant must be $\pm 1$. Hence $1\in I_{v,w}$.
\end{proof}
\section{Removing the redundant defining minors}\label{s.redminors}
\subsection{A formula for the rank matrix $\tilde{R}_w$}\label{ss.rmatrix}
In this subsection, we provide an explicit formula for the entries of the rank matrix $\tilde{R}_w$. This formula can make machine calculation easier or faster.
\begin{lemma}\label{l.rank-matrix-formula}
Let $w \in S_n$ and $q\in \{1,\dots ,n\}$. For each $q$, consider the sets:
	\begin{align*}
		A_q^{(q)}=\{w(i)|1\leq i \leq q\} \quad &and \quad  a_q^{(q)} = min A_q^{(q)}. \\
		A_{q-1}^{(q)}=A_q^{(q)}-\{ a_q^{(q)}\} \quad &and \quad  a_{q-1}^{(q)} = min A_{q-1}^{(q)}. \\
		\vdots\\
		\textnormal{In general, }A_i^{(q)}=A_{i+1}^{(q)}-\{ a_{i+1}^{(q)}\} \quad &and \quad  a_i^{(q)} = min A_i^{(q)} \ for \ 1\leq i\leq q-1.\\
	\end{align*}
	Let $p,q\in\{1,\ldots,n\}$ be arbitrary. Then, the $(p,q)$-th entry $\tilde{r}_{pq}^{(w)}$ of $\Tilde{R}_w$ is given by :
	\[
	\tilde{r}_{pq}^{(w)}=\begin{cases}
		q & \textnormal{for}\quad 1\leq p \leq a_{q}^{(q)},\\
		i-1 & \textnormal{for}\quad a_{i}^{(q)} < p \leq a_{i-1}^{(q)} \textit{ where $2\leq i \leq q$},\\
		0 & \textnormal{for}\quad p > a_{1}^{(q)}.
	\end{cases}
	\]
\end{lemma}
\begin{proof}
Obvious. 
\end{proof}
Observe that the row indexing of $Z^{(v)}$ is from bottom to top, but the row indexing of $\Tilde{R}_w$ does not follow this convention (it is from top to bottom).	
\begin{example}\label{e.discardingminors}
	Let $v,w \in S_5 $ such that $v=23451$ and $w=42531$. Correspondingly, \[
	Z^{(v)}=\begin{bmatrix}
		0 &0 &0 &0 &1\\
		1 &0 &0 &0 &0\\
		z_{31} &1 &0 &0 &0\\
		z_{21} &z_{22} &1 &0 &0\\
		z_{11} &z_{12} &z_{13} &1 &0
	\end{bmatrix}\quad \textnormal{and} \quad \tilde{R}_w=\begin{bmatrix}
		1 &2 &3&4 &5\\
		1&2&3&4&4\\
		1&1&2&3&3\\
		1&1&2&2&2\\
		0&0&1&1&1
	\end{bmatrix}
	\]
	For the 3rd column of $\Tilde{R}_w$, we have $q=3$, and
	\begin{align}\notag
		A_3^{(3)}=\{2,4,5\} &\textnormal{ and } a_3^{(3)}=2,\\\notag
		A_2^{(3)}=\{4,5\} &\textnormal{ and } a_2^{(3)}=4,\\\notag
		A_1^{(3)}=\{5\} &\textnormal{ and } a_1^{(3)}=5.\\\notag
	\end{align}
	Now, we have the entries $\tilde{r}_{p3}^{(w)}$ of $\Tilde{R}_w$ as follows :
	\[
	\tilde{r}_{p3}^{(w)}=\begin{cases}
		3 & \textnormal{for}\quad 1\leq p \leq a_{3}^{(3)},\\
		i-1 & \textnormal{for}\quad a_{i}^{(3)} < p \leq a_{i-1}^{(3)} \textit{ where $2\leq i \leq 3$},\\
		0 & \textnormal{for}\quad p > a_{1}^{(3)}.
	\end{cases},\]
	which, using the above calculation, simplifies to :\[
	\tilde{r}_{p3}^{(w)}=\begin{cases}
		3 & \textnormal{for}\quad 1\leq p \leq 2,\\
		2 & \textnormal{for}\quad 2=a_{3}^{(3)} < p \leq a_{2}^{(3)}=4, \ i.e \ , p\in \{3,4\}\\
		1 & \textnormal{for}\quad 4=a_{2}^{(3)} < p \leq a_{1}^{(3)}=5,\ i.e \ ,p=5\\
	\end{cases}
	\]
	One can now see that these $\tilde{r}_{p3}^{(w)}$ as $1\leq p\leq 5$, coincides with the third column of $\tilde{R}_w$.\\
	\end{example}
\subsection{Reducing certain defining minors}\label{ss.redminors}
\begin{remark}\label{r.fulton}
	In \cite{Ful92}, he has worked with the northwest subminor of a matrix similar to $Z^{(v)}$ and there (lemma 3.10 to be specific), he has defined something called ``essential set" and proved that it is sufficient to consider only certain northwest matrices (corresponding to the ``essential set" ) and work with their minors. This approach can also be used with our convention to get a much optimal set of generators for $I_{v,w}$. 
\end{remark}
We can define the ``essential set" in our convention as follows:\\
Fix a $w \in S_n$. Draw a $n\times n$ box where for each $1\times 1$ sub-box, the row and column it is in, be its labeling. Then, for each $i\in \{1,2,\dots,n\}$, consider the box $(w(i),i)$. Put a dot in it and draw a line from the dot in $(w(i),i)$ to its right and top,i,e, draw lines crossing the boxes $\{(w(i),j)|j>i\}$ and $\{(s,i)|s<w(i)\}$ respectively. For each $i$, repeat the process. Then, consider the remaining boxes that don't have a dot or a line through it. Let $A_w$ be the set of their labelings. Let $E_w=\{(i,j)\in A_w \ |\textit{ neither } (i-1,j)\in A_w \textit{ nor } (i,j+1)\in A_w\}$. Then $E_w$ is the ``\textbf{essential set}" in our convention corresponding to $w$.
\begin{remark}\label{r.essentialset}
One can prove a lemma similar to lemma 3.10 in \cite{Ful92}, that in order to find the generators of $I_{v,w}$, it suffices to take only the $\tilde{r}_{st}^{(w)}+1$ minors corresponding to the Southwest $(n-s+1)\times t$ submatrix of $Z^{(v)}$ for each $(s,t)\in E_w$, where $E_w$ is the essential set as above (in our convention). These minors are called the \textbf{defining essential minors} of $I_{v,w}$.
\end{remark}

\section{Some results on homogeneity}\label{s.resonhom}
\textbf{Fact 1:} Any square matrix $[X_{ij}]_{n\times n}$, where each $X_{ij}$ is either $0$ or an indeterminate, has determinant equal to a homogeneous polynomial.\label{fact1}
\begin{proof}
	For a $2\times 2$ matrix, depending on the position of number of zeroes it has, it is either a degree 2 homogeneous polynomial or $0$. Assume [$X_{ij}$] be a $n\times n$ matrix where each $X_{ij}$ is either a zero or an indeterminate. By induction, all the $(n-1)\times (n-1)$ minors can be assumed to be a homogeneous polynomial of degree $(n-1)$ or 0. Taking the determinant of the matrix [$X_{ij}$] expanding via, say, the first row, we get a summation of $X_{1j}$.(a $(n-1)$ minor) for each j$j$. But each term in the summation ,depending on the $X_{1j}$ or the minor, is either 0 or becomes a $n$ degree polynomial and thus yields a degree $n$ polynomial as a whole or 0. Hence we get the result.
\end{proof}
\begin{lemma}\label{l.homthm}
	(\textbf{The homogeneity lemma:}) Consider the polynomial ring $\mathbb{C}[x_1, x_2, \dots x_l]$ for some $l \in \mathbb{N}$. Let $I$ be an ideal in it such that $I=\langle f_1, f_2, \dots, f_n\rangle$, where $m$ of the $f_i$ are inhomogeneous and rest are homogeneous, $m\leq n$. Let $f_i$ be an inhomogeneous polynomial such that $f_i=\sum_{j=1}^{n_i}h_{ij}$, where $h_{ij}$ are the homogeneous components of $f_i$. Then,\\
	$I$ is homogeneous $\implies$ For every $f_i$ inhomogeneous, there exists some $h_{ij}\ (1\leq j\leq n_i)$, such that every monomial in that $h_{ij}$ is divisible either by at least one monomial of some $f_k$ ($k\neq i$) or by at least one monomial of some $h_{ik}$ ($k\in\{1,\ldots,n_i\}\setminus\{j\}$).
\end{lemma}
\begin{proof}
	We may assume $f_i$ to be $f_1$ by relabeling if necessary. Let $I$ be homogeneous. Then each $h_{1j} \in I$, in particular $h_{11}\in I$. Say, $h_{11}=\sum_{i=1}^{n}g_if_i$ for some polynomials $g_i \in \mathbb{C}[x_1, x_2, \dots x_l] $.\\
	\textbf{Case (1):} $g_1$ is a constant polynomial $(\neq 1)$. Say, $g_1 = c (\neq 1)$. \\
	Then we have:\\
	$h_{11}=\sum_{i=2}^{n}g_if_i + c(h_{11}+\dots+ h_{1n_1})$\\
	$\implies h_{11}= \sum_{i=2}^{n}g_i'{f_i} + c'(h_{12}+\dots+ h_{1n_1})$ \quad \quad \quad (1)\\
	where $g_i'=\frac{g_i}{1-c}$ and $c'=\frac{c}{1-c}$.\\
	\textbf{Subcase (a):} $c'=0$.\\
	Then $h_{11}$= $\sum_{i=2}^{n}g_i'{f_i} $.\quad \quad \quad $(*)$\\
	Let $X_{11}$ be an arbitrary term of $h_{11}$. Then equation $(*)$ implies that $X_{11}$ is a sum of some terms of the form $m_{ab}X_{ac}$, where $m_{ab}$ is a term in some $g_a$ ($a\neq 1$), and $X_{ac}$ is a term in some $f_a$ ($a\neq 1$).\\
	Now, $X_{11}$ is a monomial, which implies that at least one term of the form $X_{ac}$ ($a\neq 1$) must divide $X_{11}$. This proves the lemma in this subcase.\\
	\textbf{Subcase (b):} $c'\neq$ 0.\\
	Then, for equation $(1)$ to hold, the terms of $c'(h_{12}+\ldots+h_{1n_1})$ must get canceled with some terms from $\sum_{i=2}^{n}g_i'{f_i} $. Let $X_{1j}$ be one arbitrary term from the part $c'(h_{12}+\dots+ h_{1n_1})$. Then we must have that:
	$$-X_{1j}=\mbox{a sum of some terms of the form}\ m_{ab}X_{ac},$$
	where $a\neq 1$, and $m_{ab},X_{ac}$ are as in Subcase (a) of Case 1 above. \\
	Since $X_{1j}$ is a monomial, therefore the above equation implies that at least one term of the form $X_{ac}$ ($a\neq 1$) divides $X_{1j}$. This proves the lemma in this subcase also.\\
	\textbf{Case (2):} $g_1 = 1$.\\
	Then, we have $\sum_{i=2}^{n}g_if_i + (h_{12}+\dots+ h_{1n_1})=0$. \\
	This implies that
	$$h_{12}=-\sum_{i=2}^{n}g_if_i+(h_{13}+\cdots+h_{1n_1}).$$
	Since $h_{13},\ldots,h_{1n_1}$ are distinct from $h_{12}$, they cannot contribute to $h_{12}$. Hence the terms of $h_{12}$ come from the portion $-\sum_{i=2}^{n}g_if_i$ only. Let $X_{12}$ be an arbitrary term of $h_{12}$. Then by a similar argument as in Case (1), we are done.\\
	\textbf{Case (3):} $g_1$ doesn't have any constant term. \\
	We have 
	$$h_{11}=\sum_{i=2}^{n}g_if_i+g_1(h_{11}+\cdots+h_{1n_1}).$$
	Notice that $g_1h_{11}$ cannot give a term of $h_{11}$ (as $g_1$ doesn't have any constant term). Therefore, we must have that: $h_{11}$ is obtained from the part
	$$\sum_{i=2}^{n}g_if_i+g_1(h_{12}+\cdots+h_{1n_1}).$$
	Let $X_{11}$ be an arbitrary term of $h_{11}$. Then
	$$X_{11}=\mbox{a sum of some terms of the form}\ m_{ab}X_{ac}\ \mbox{where}\ a\neq 1$$
	$$\mbox{and some terms of the form}\ m_{1r}X_{1j},$$
	where $m_{1r}$ is a monomial in $g_1$, $X_{1j}$ is a monomial from the portion $h_{12}+\cdots+h_{1n_1}$, and $m_{ab},X_{ac}$ are as in Case (1).\\
	Now, $X_{11}$ is a monomial. So we must have that at least one of the $X_{ac}$ ($a\neq 1$) or $X_{1j}$ ($j\neq 1$) divides $X_{11}$. Thus, we are done in this case.\\
	\textbf{Case (4):} $g_1$ contains a non zero constant term (say c).\\
	We had
	$$h_{11}=\sum_{i=2}^{n}g_if_i+c(h_{11}+\cdots+h_{1n_1})+$$
	$$\mbox{a sum of some non constant terms multiplied to the}\ h_{1j}\mbox{'s}.$$
	\textbf{Subcase (a):} $c=1$.\\
	Then for some $k(\neq 1)$, 
	$$-h_{1k}=\sum_{i=2}^{n}g_if_i+(h_{12}+\cdots+h_{1,k-1}+h_{1,k+1}+\cdots+h_{1n_1})$$
	$$+\mbox{a sum of some non constant monomials multiplied to the}\ h_{1j}\mbox{'s}.$$
	Since a non constant monomial multiplied to $h_{1k}$ cannot give any term of $h_{1k}$ and all the other $h_{1j}$ for $j$ such that $2\leq j\leq n_1,j\neq k$ have different degrees, therefore $-h_{1k}$ is obtained from the part 
	$$\sum_{i=2}^{n}g_if_i+(g_1-1)(h_{11}+h_{12}+\cdots+h_{1,k-1}+h_{1,k+1}+\cdots+h_{1n_1}).$$
	Let $X_{1k}$ be an arbitrary term of $-h_{1k}$. Then the above argument implies the lemma in this subcase.\\
	\textbf{Subcase (b):} $c\neq 1$.\\
	Then
	$$h_{11}=\sum_{i=2}^{n}g_i'f_i+c'(h_{12}+\cdots+h_{1n_1})$$
	$$+\mbox{a sum of some non constant monomials multiplied to the}\ h_{1j}\mbox{'s},$$
	where $g_i'=\frac{g_i}{1-c}$ and $c'=\frac{c}{1-c}$.\\
	Then $h_{11}$ is obtained from the part 
	$$\sum_{i=2}^{n}g_i'f_i+\mbox{a sum of some non constant monomials multiplied to the}\ h_{1j}\mbox{'s}\ (j\neq 1).$$
	This proves the lemma in this subcase.
\end{proof}
\begin{example}\label{e.homthm}
	Let I=$\langle f_1,f_2,f_3\rangle$ be an ideal in $\mathbb{C}[x_1,x_2,x_3,x_4]$ where $f_1=x_1x_2,f_2=x_2x_3$ be homogeneous polynomials and $f_3=x_1x_3+x_2x_3x_4$ be a inhomogeneous polynomial. Then clearly (by inspection) one can observe that $I=\langle f_1,f_2,x_1x_3\rangle$, hence, $I$ is homogeneous and one of the components of the inhomogeneous part $f_3$, namely $x_2x_3x_4$, is divisible by one of the $f_i \in \{f_1,f_2\}$, namely, by $f_2=x_2x_3$.
\end{example}
\section{Determinants and Paths}\label{s.paths}
\subsection{Paths}\label{ss.paths}
Let X=[$x_{ij}$] be a square matrix of size $q \in \mathbb{N}$.\\
A \textbf{path in X} is a finite sequence of points $\{x_{i_1j_1},\dots ,x_{i_qj_q}\}$ such that \\
(1) $j_k=k \ \forall \ k\in \{1,\dots ,q\}.$\\
(2) $\{i_1,\dots , i_q\}$ are all distinct.\\
\begin{remark}\label{r.pathremark1}
	We can draw a path in X as a finite sequence of line segments joining the consecutive points in the set $\{x_{i_1j_1},\dots ,x_{i_qj_q}\}$.
\end{remark}
\begin{remark}\label{r.pathremark2}
	A path in X has the property that no two points in it can be joined by a vertical or horizontal line segment.
\end{remark}
\begin{defi}\label{d.pathdef}
	A \textbf{non zero path} in X is a path in X which doesn't contain any zero element. 
\end{defi}
\begin{defi}\label{d.subpathdef}
	A \textbf{subpath} of a path in X is a subcollection of points of the path in X. 
\end{defi}
\begin{example}\label{e.path}
	Let \[A=\begin{bmatrix}
		a &b &c\\
		d &e &f\\
		g &h &i
	\end{bmatrix}
	\]
	Then \textbf{aei} is a path in A, so is \textbf{gbf}, but \textbf{abc} isn't. Also, \textbf{g} or \textbf{gb} are subpaths of \textbf{gbf}. One can also see that detA is a sum of all of its possible paths with a plus or minus sign upfront.
\end{example}
\textbf{Observation 1:}  Let $A$ be any $p \times p$ ($p \geq 2$) minor of $Z^{(v)}$  with rows $\{i_1 < \dots < i_p\}$ and columns $\{j_1 < \dots < j_p\}$ in relation to $Z^{(v)}$ (rows counted from bottom to top and columns counted from left to right). If $A$ is non singular, then $detA$ is either $\pm 1$, or a sum of non constant monomials.\label{o.o1}
\begin{proof}
	Let $A$ be non singular such that $detA$ is not a sum of non constant monomials. That is, Suppose $detA$ is a polynomial with a constant term. Then, by definition of $Z^{(v)}$, that constant term must be $1$ or $-1$ and this means, in $A$, there is an 1 in each column. Let $\alpha$ be a non zero path in $A$ and the 1 in the $1^{st}$ column of $A$ be in row $i_{\alpha_1}$. Since all the entries to the right of 1 is zero and it being the first column in $A$, the non zero entry that $\alpha$ has to pick from the $i_{\alpha_1}^{th}$ row, must be 1. Let the 1 in 2nd column be at row $i_{\alpha_2}$. Clearly $\alpha_1 \neq \alpha_2$. Again, since all the entries to the right of this 1 (if any) are 0, and it cannot pick anything from left of 1 as it has already picked its entry (i.e 1) from the first column, $\alpha$ must pick 1 from the $i_{\alpha_2}^{th}$ row. Iterating the same argument for each column, we see that $\alpha$ must pick 1 from each column and thus $\alpha=\pm 1$ is the only possible path,i.e, $detA=\pm 1$. This completes the proof.
\end{proof}
\textbf{Observation 2:} Let $A$ be any $p \times p$ ($p \geq 2$) minor of $Z^{(v)}$  with rows $\{i_1 < \dots < i_p\}$ and columns $\{j_1 < \dots < j_p\}$ (rows counted from bottom to top and columns counted from left to right). Let $f$ and $g$ be any two non constant monomials in $det A$. Then neither of them divides the other one.\label{o.o2}
\begin{proof}
	Let $g$ divide $f$ ($g$ $\neq$ $f$), both non constant. $f$ be a path of the form : $\Pi z_{i_{\alpha_1}j_1}\cdot \dots \cdot z_{i_{\alpha_p}j_p}$ where $z_{i_{\alpha_k}j_k}$ (for $1\leq k\leq p$) is either an indeterminate or $1$.
	Let $j_s (1\leq s \leq p-1$) be the maximal index such that the variables of $g$ equals with variables of $f$ upto that index, and so, $(s+1)^{th}$ variable of $g$ doesn't equal the $(s+1)^{th}$ variable of $f$. So, say, $g$ is of the form $\Pi z_{i_{\alpha_1}j_1}\cdot \dots \cdot z_{i_{\alpha_s}j_s}***$   where $*$ stands for the later variables of $g$ corresponding to the columns $j_{s+1}$ to $j_p$. Since $g$ is also a nonzero path, so it should have picked a non zero entry from the $j_{s+1}$-th column, and as $g$ (a nonzero path) is supposed to divide the path $g$, it must be that the entry that $g$ has picked from the $j_{s+1}$-th column is $1$. Say, it belongs to the $i_o^{th}$ row. But then by the nature of $Z^{(v)}$, every entry in the $i_o^{th}$ row, to the right of $1$ (if any) must be $0$. Now, $f$ being a nonzero path, must have picked a nonzero entry from the $i_o^{th}$ row and that entry must have come from a column left of $1$ i.e from a column from $j_1$ to $j_s$ (Note that the path $f$ has not picked the entry $1$ from the $i_0$-th row, because of the definition of $j_s$). But this is impossible as $g$ agrees with $f$ in all those columns and this would then mean $g$ has picked two entries from the $i_o^{th}$ row. \\
	Since all the steps above are consistent with one another, it must be that our starting assumption was wrong. So, $\nexists$ any such  non constant $g$ that divides $f$.
\end{proof}
\begin{remark}\label{r.lemma-homthm}
It follows from \hyperref[o.o2]{Observation 2} above that the conclusion:
\begin{quote}
$I$ is homogeneous $\implies$ For every $f_i$ inhomogeneous, there exists some $h_{ij}\ (1\leq j\leq n_i)$, such that every monomial in that $h_{ij}$ is divisible either by at least one monomial of some $f_k$ ($k\neq i$) or by at least one monomial of some $h_{ik}$ ($k\in\{1,\ldots,n_i\}\setminus\{j\}$).
\end{quote}
in Lemma \ref{l.homthm} can be replaced by the following:
\begin{quote}
$I$ is homogeneous $\implies$ For every $f_i$ inhomogeneous, there exists some $h_{ij}\ (1\leq j\leq n_i)$, such that every monomial in that $h_{ij}$ is divisible by at least one monomial of some $f_k$ ($k\neq i$). 
\end{quote}
in case $I$ is a KL-ideal.\\
The contrapositive of this statement will be:
\begin{quote}
 If there exists a $f_i$ inhomogeneous such that given any $h_{ij}\ (1\leq j\leq n_i)$, there exists a monomial in $h_{ij}$ which is not divisible by any monomial of any $f_k\ (k\neq i)$, then $I$ is inhomogeneous.
\end{quote}
We will need this contrapositive statement very much later in this paper.
\end{remark}
Let us now have some results on division of terms.
\begin{theorem}\label{t.poldiv1}
	Let $A$ and $B$ be 2 minors in $Z^{(v)}$ of size $p$ and $q$ respectively. Firstly, suppose $p\le q$ and $A$ is not a subminor of $B$. Let a monomial in $det A$ be called $m_A$ and similarly $m_B$ relative to $detB$. Let A and B have the rows $i_1\dots,i_m$ and columns $j_1,\dots,j_n$ in common. Then $\exists \ m_A$ such that $m_A|m_B \iff$ all the rows and columns of $A$ that aren't equal to $i_s$ for some $1\leq s \leq m$ or $j_k$ for some $1\leq k \leq n$, are of the form $(*,*,\dots,1,0,\dots,0)$ and $(0,\dots,0,1,*,\dots,*)^t$ respectively.
\end{theorem}
\begin{proof}
	$(\implies)$ Suppose $\exists \ m_A$ such that $m_A|m_B$. Then, the path $m_A$ has picked 1 from all the rows and columns that don't appear in the expression of $m_A$. Such rows and columns exist as $A$ is not a sub-minor of $B$(as then $\exists$ at least one row and column that is not a row or column in $B$ and $m_A$ being a path, that row and column has some contribution in $m_A$). Therefore, by definition of $Z^{(v)}$,those rows and columns are of the mentioned forms.\\
	($\impliedby$) Let $m_B$ be of the form : $\Pi * \dots * z_{i_{\alpha}j_{\beta}}*\dots*$ where $z_{i_{\alpha}j_{\beta}}$ denote the variables corresponding to the intersecting rows and columns and $*$ denote the variables for the non intersecting rows and columns. If all the $z_{i_{\alpha}j_{\beta}}=1$, then, by using our hypothesis, we can form a path $m_A=1$ and clearly $m_A|m_B$.\\
	Else, say not all $z_{i_{\alpha}j_{\beta}}=1$ and let, $\alpha_1,\dots, \alpha_t$ and $\beta_1,\dots,\beta_{t^{'}}$ be the indices of these $z_{i_{\alpha}j_{\beta}}$. Then, it must be that $z_{i_{\alpha}j_{\beta}}=1$ for all $\alpha\in \{1,\dots,m\}\setminus \{\alpha_1,\dots, \alpha_t\},\beta\in \{1,\dots,n\}\setminus \{\beta_1,\dots,\beta_{t^{'}}\}$ and as these are common to $A$ also, we can pick 1 from these rows and columns too. This together with our assumption, helps us to form the path $m_A=\prod\limits_{\substack{\alpha\in\{\alpha_1,\dots, \alpha_t\} \\ \beta\in\{\beta_1,\dots,\beta_{t^{'}}\}}}^{}z_{i_{\alpha}j_{\beta}}$ and clearly $m_A|m_B$.
	\end{proof}
	
	\begin{theorem}\label{t.poldiv2}
		Assume the same notations as Theorem \ref{t.poldiv1} and that A is a sub-minor of B. Then $\exists m_A$ such that $m_A|m_B$ $\iff$ $\exists$ a path in A that is a sub-path of $m_B$.
	\end{theorem}
	\begin{proof}
		$\exists m_A$ such that $m_A|m_B \implies m_A$ is a path in $A$ that is a sub-path of $m_B$.\par 
		
		On the other hand, if $\exists $ a path in $A$ that is a sub-path of $m_B$, take the path in $A$ and call it $m_A$. Clearly, $m_A|m_B$.
	\end{proof}
	\begin{theorem}\label{t.poldiv3}
		Assume the same as above except $p>q$. Then we notice that A cannot be a sub-minor of B. Keeping all the other notations same, we have then $\exists \ m_A$ such that $m_A|m_B \iff$ all the rows and columns of $A$ that aren't equal to $i_s$ for some $1\leq s \leq m$ or $j_k$ for some $1\leq k \leq n$, are of the form $(*,*,\dots,1,0,\dots,0)$ and $(0,\dots,0,1,*,\dots,*)^t$ respectively.
	\end{theorem} 
	\begin{proof}
		The proof is same as Theorem \ref{t.poldiv1} as we never used the part $p\leq q $ and as $p>q$, we don't have the case of $A$ being a sub-minor of $B$ and the rest of the proof for Theorem \ref{t.poldiv1} follows.
	\end{proof}
Let us now go back to paths and non singularity of a determinant in $Z^{(v)}$.
\begin{lemma}\label{l.zeroRorC}
	Let A be any $p \times p$ ($p \geq 2$) minor of $Z^{(v)}$. Then, A is non singular if and only if $\exists$ at least one non zero path in A. 
\end{lemma}
\begin{proof}
	$\implies$ obvious.\\
	$\impliedby$ We will show that any two arbitrary and distinct nonzero paths $\alpha$ and $\beta$ in $A$ do not cancel with each other.\par 
	Since $\alpha$ and $\beta$ are distinct, they must disagree on their entries in some column, say column $j$. Let us say that from the column $j$, $\alpha$ picks $z_{ij}$ and $\beta$ picks $z_{i'j}$ (both $z_{ij}$ and $z_{i'j}$ being two nonzero entries of $A$). Now, from the definition of the $Z^{(v)}$ matrix, we have that $z_{ij}\neq z_{i'j}$ (and both cannot be $1$ at the same time). Hence $\alpha$ cannot cancel out with $\beta$. Since $\alpha$ and $\beta$ were arbitrary and distinct, it follows that no two such nonzero paths can cancel out with each other. Hence $A$ is non singular.
\end{proof}
Now comes one of the main theorems of this section, which provides a necessary and sufficient condition for $det A$ to be inhomogeneous, where $A$ is any $p\times p$ ($p\geq 2$) minor of $Z^{(v)}$.
\begin{theorem}\label{t.hompaththm}
	(\textbf{The inhomogeneous determinant theorem:}) Let A be any $p \times p$ ($p \geq 2$) minor of $Z^{(v)}$. Then det$A$ is inhomogeneous if and only if A contains at least one column of the form $(0\dots 0 , 1 , \dots z \dots)^t$ such that $\exists$ a non zero path which passes through z, where z $\in \{z_{ij} | 1\leq i \leq n, 1\leq j \leq n\} $ and $\vec{v}^t$ denotes the transpose of the vector $\vec{v}$. [Note here that the number of zeroes in the beginning of the column $(0\dots 0,1,\dots z \dots)^t$ can also be zero.]
\end{theorem}
\begin{proof}
	$\impliedby$ Let the column (of the form $(0,\ldots,0,1,\ldots,z,\ldots)^t$) under consideration be the $j$-th column of $A$, and let the $1$ in that column correspond to the $i_0$-th row of $A$. Notice that $j\neq 1$, because if $j=1$, then there will not exist any nonzero path through any variable entry $z$ lying below $1$ (as such a path will pick only $0$ from the row $i_0$). \par 
	Let the entry $z$ be in the $i$-th row of the $j$-th column of $A$. Clearly $i<i_0$, due to the structure of the column $j$. By assumption, there exists a nonzero path in $A$ passing through $z$ (which is placed at the $(i,j)$-th cell of $A$). Let $\alpha$ be such a path. Then $\alpha$ has picked $z$ from the $(i,j)$-th cell of $A$. Since $\alpha$ is a nonzero path, it must have picked a nonzero entry from the $i_0$-th row of $A$, call it $z_{i_0,j_0}$. Since the row $i_0$ can contain a unique $1$ entry, therefore $z_{i_0,j_0}$ is a nonzero variable entry. \par 
	If $j_0>j$, then $\alpha$ has to pick a zero, which is not possible as $\alpha$ is a nonzero path. If $j_0=j$, then $z_{i_0,j_0}=z_{i_0,j}=1$. But $\alpha$ has already picked $z$ from the $(i,j)$-th cell of $A$, so that gives us a contradiction. Hence we must have that $j_0<j$.\par 
	The entry $z_{i,j_0}$ of $A$ is a nonzero variable (because, if it is $1$, then all entries to its right and top are zeroes, which is impossible. And if it is $0$, then there exists a $1$ either to the left of it or below it, which forces the entries $z$ or $z_{i_0,j_0}$ to be zero, a contradiction). \par 
	Now, we can construct another path $\beta$ in $A$ by switching the entries $z_{i_0,j_0}$ and $z$ of path $\alpha$ with the entries $1$ (at the $(i_0,j)$-th cell of $A$) and $z_{i,j_0}$ of $A$ respectively. The resulting path $\beta$ is then a monomial of degree exactly $1$ less than that of $\alpha$.  But both $\alpha$ and $\beta$ appear in $det A$ and they do not cancel out with any other path in $A$ (due to the proof of Lemma \ref{l.zeroRorC}). Hence $det A$ is inhomogeneous. \par
	$\implies$ In this part, we will prove the contrapositive statement. That is, we will prove that if the condition (given in the statement of the theorem) doesn't hold, then $det A$ is homogeneous.\par 
	Suppose the condition (of the theorem) doesn't hold. Then there doesn't exist any column of the given form. This implies that one of these $3$ cases can occur:\\
	\noindent\textbf{Case I:} There is no column in $A$ having an entry $1$ in it.\\
	In this case, we are done by \hyperref[fact1]{Fact 1}. \\
	\noindent\textbf{Case II:} In all the columns in $A$ having an entry $1$ in them, all the entries below the entry $1$ are $0$.\\
	Look at such a column. Say, it is column $j$. Let $1$ be placed in the $i$-th row of the column $j$. Expand $detA$ along column $j$. Then $det A$ equals the determinant of a $(p-1)\times (p-1)$ submatrix of $A$ formed by removing the $i$-th row and the $j$-th column of $A$. Look at this $(p-1)\times (p-1)$ submatrix of $A$, call it $A'$. If $A'$ also contains a column of the type $(0,\ldots,0,1,0,\ldots,0)^t$ in it, then repeat the same procedure to form a smaller submatrix. Finally, after removing all columns of the type $(0,\ldots,0,1,0,\ldots,0)^t$ from $A$, we are left with a square submatrix of $A$ of size strictly less than $p$. This submatrix has no column having a $1$ in it. Then we appeal to \hyperref[fact1]{Fact 1} and we are done.\\
	\noindent\textbf{Case III:} There exist column(s) in $A$ of the form $(0,\ldots,0,1,\ldots,z,\ldots)^t$, but there does not exist any nonzero path in $A$ through any nonzero variable entry below\ $1$.\\
	Say, the column $j$ of $A$ is some such column and the entry $1$ in this column $j$ is at the $i$-th row. Since there does not exist any nonzero path through any nonzero variable entry below $1$, it follows that $det A=det A^0$, where $A^0$ is the $(p-1)\times(p-1)$ submatrix of $A$ obtained by removing the $i$-th row and the $j$-th column from $A$.\par
	In $A^0$, if a similar type of column exists, then repeat the same procedure. Stop the procedure when there is no such column left in the smaller submatrix formed in this manner. Clearly, this procedure will stop after a finite stage and we will be left with a square submatrix of $A$ of size strictly smaller than $p$, which has no column of this type. Such a submatrix falls into either case I or case II and hence we are done.
\end{proof}
\begin{example}\label{e.hompaththm}
	Consider the matrix \[A=\begin{bmatrix}
		z_{31} &1 &0\\
		z_{21} &z_{22} &z_{23}\\
		z_{11} &z_{12} &z_{13}
	\end{bmatrix}\]
	Then, $detA= z_{31}z_{22}z_{13}- z_{31}z_{12}z_{23}-z_{11}z_{23}+z_{13}z_{21}$ which is clearly an inhomogeneous polynomial.
\end{example}

For computational convenience, we now provide certain conditions on $v$ (which we call \textbf{conditions} $\delta_i$), which tell us when does there exist a nonzero path in a $p\times p$ (where $p\geq 2$) minor $A$ of $Z^{(v)}$. But to state those $\delta_i$ conditions, we use (for simplicity) a change in our convention, which we state below first:\\ \\ 
\noindent\textbf{A change in convention:} Recall that according to our convention of $Z^{(v)}$ (which we call the old convention), for any $v\in S_n$, we are taking the value $v(j)$ for each $j\in\{1,\ldots,n\}$ counting from top to bottom. For example, say in $S_7$, $v(1)=4$ means that the $1$ in the first column of the $Z^{(v)}$ matrix lies in the $4$-th row from the top, whereas the rows in the $Z^{(v)}$ matrix are labeled from bottom to top.
\par
We will use a different convention for $v(j)$ \textbf{in the rest of this section}, which is as follows:
\begin{quote}
If $v(j)=i$ for some $i,j\in\{1,\ldots,n\}$ in the old convention, then in the new convention, it would mean $v(j)=n-i+1$. That is, we are counting the row number corresponding to $v(j)$ from bottom to top. 
\end{quote}
One can readily see that the new convention is equivalent to the old convention because, we can replace the $v(j)$'s written below (in the rest of this section) by $n-v(j)+1$'s. We are using this new convention because it simplifies the statements below and avoids unnecessary cluttering of symbols.\par 
Notice that if we denote a permutation in $S_n$ as $v$ (in the new convention), then it is equivalent to the permutation $v'$ in the old convention, where $v'(j)=n-v(j)+1$ for all $1\leq j\leq n$. Using this identification $v\leftrightarrow v'$, one can see that if we discard some $v$ following the new convention (in steps 6 and 7 of the algorithm mentioned in \S \ref{s.whichwtodiscard}), then the corresponding $v'$ has been discarded in the old convention. \\
\noindent\textbf{Caution:} Notice that the above change of convention for $v$ (or $v(j)$'s rather) doesn't affect the convention of $v^{-1}$. This is because the order of labeling of the columns of $Z^{(v)}$ is from left to right.
\begin{lemma}\label{l.LD}
	Let $A$ be as in \hyperref[o.o2]{Observation 2}, with rows $\{i_1 < \dots < i_p\}$ and columns $\{j_1 < \dots < j_p\}$ w.r.t $Z^{(v)}$. Then $A$ has one row or column zero if and only if it satisfies at least one of the following properties related to $v\in S_n$. \\
	(1) $v(j_k) < i_1 $ for some $1\leq k \leq p$.\\
	(2) $v^{-1}(i_l) < j_1$ for some $1\leq l \leq p$.\\
	(3) min$\{v(j_k) \ | \ 1\leq k \leq p)\} \geq i_1$ and $\exists\ k\in\{1,\ldots,p\}$ such that 
		$i_s<v(j_k)<i_{s+1},1\leq s \leq p-1$ $(or \ v(j_k)>i_p)$ and $v^{-1}(i_\xi)<j_k \ \forall \ 1\leq \xi \leq s $ $(or \ v^{-1}(i_\xi) < j_k \ \forall \ 1\leq \xi \leq p)$ respectively. \\
	(4) min$\{v^{-1}(i_s) \ | \ 1\leq s \leq p)\} \geq j_1$ and $\exists$ s such that 
	$j_\alpha<v^{-1}(i_s)<j_{\alpha+1},1\leq \alpha \leq p-1$ $(or \ v^{-1}(i_s)>j_p)$ then $v(j_\beta)<i_s \ \forall \ 1\leq \beta \leq \alpha $ $(or \ v(j_\beta) < i_s \ \forall \ 1\leq \beta \leq p)$ respectively.\\
Here, (1) and (3) are conditions for when there is a zero column and (2) and (4) are conditions for when there is a zero row.
\end{lemma}
\begin{proof}
	$\implies$ Suppose one column is zero. This can happen in 3 ways. First, when for any of the columns $j_1,\dots ,j_p$, $v(j_k)$ (for any $1\leq k\leq p$) lies below all the rows $i_1,\dots,i_p$ and one can see that in this case it suffices that $v(j_k) < min\{i_1,\dots ,i_p\}=i_1$. This is (1) in our lemma.\par 
	Now, secondly, it can happen that $v(j_k)$ for some $k\in\{1,\ldots,p\}$ lies between any two rows in $\{i_1,\dots,i_p\}$ , i.e, in some row in the original $Z_{s,t}^{(v)}$ matrix for some $s,t$ (it does not lie on any of the rows $\{i_1,\dots,i_p\}$ as else the column won't be zero in $A$) and let $\{i_1,\dots,i_s\}$ be the set of rows below it relative to $A$. Then, all the rows above, i.e, $i_{s+1},\dots, i_p$, are 0 at the $j_k^{th}$ column. Now, if for all the rows $i_{\xi}, 1\leq \xi \leq s$, $v^{-1}(i_{\xi})$ lies to the left of $j_k$, then by definition of $Z^{(v)}$, all the rows $i_1,\dots,i_s$ are also 0 in the $j_k^{th}$ column. This will make the whole $j_k^{th}$ column 0 in terms of $A$. Similar reasoning gives the third case that if $v(j_k)>i_p$, then  $v^{-1}(i_\xi) < j_k \ \forall \ 1\leq \xi \leq p$ makes the entire column zero. We have joined the second and third way together due to their similarity and labeled it as condition (3) in our lemma. So, one can see that for a column to be zero, we must have at least one of (1) or (3) must hold. Similar reasoning shows that for a row to be zero, at least one of (2) or (4) must hold.\\
	$\impliedby$ If (1) or (3) holds, we have a zero column in $A$, if (2) or (4) holds, we have a zero row in $A$. So, if any of the four points hold, we have either a zero column or a zero row hereby proving the lemma in this direction.
\end{proof}
Now, even if $\nexists$ any zero row or column, we can still have $A$ to be singular,i.e, the scenario where there doesn't exist any non zero path even though there isn't any zero row or column. In order to address this, consider $A$ as above and that none of the conditions of Lemma \ref{l.LD} hold, i.e $A$ has no zero row or zero column.\par
Consider the first column $j_1$. There exists at least one non-zero entry. Pick one such entry. Call it $z_{\alpha_1,j_1}$. In order for us to have a non zero path through $z_{\alpha_1,j_1}$, we must have a non-zero entry at column $j_2$ in any row other than $\alpha_1$. Now, the condition for not having a non zero entry at the $j_2^{th}$ column (except maybe at $\alpha_1^{th}$ row) is that at least one of the following conditions should hold:\\
(1) $v(j_2) > i_p$ and $v^{-1}(i_s) < j_2 \ \forall\ i_s\neq\alpha_1$.\\
(2) $v(j_{2})=i_{s_0} (for\ some\ s_0\ such\ that\ p\geq s_0 > 1)=\alpha_1$ and $v^{-1}(i_s) < j_{2} \ \forall\ i_s\in\{i_1,\ldots,i_{s_0-1}\}\setminus\{\alpha_1\}$ or $v(j_2)=i_1=\alpha_1$ .\\
(3) $i_{s_0} < v(j_2) < i_{{s_0 + 1}}$, for some $s_0 \in \{1,\dots,p-1\}$ and $v^{-1}(i_s) < j_2 \ \forall\ i_s\in\{i_1,\ldots,i_{s_0}\}\setminus\{\alpha_1\}$. \\ \\
Note here that the condition $v(j_2)<i_1$ is not included above (in conditions (1),(2),(3)) because $i_1$ is the bottommost row in $A$.\par 
Let $i\in\{1,\ldots,p-1\}$. Let $\alpha_k$ (for $2\leq k\leq i$) be defined recursively in the same way as $\alpha_1$. For a general $j_{i+1}^{th}$ column, the condition for not having a non zero entry at the $j_{i+1}^{th}$ column (except maybe at $\alpha_k^{th}$ rows, $1\leq k \leq i$) is that at least one of the following conditions should hold:\\
(1) $v(j_{i+1}) > i_p$ and $v^{-1}(i_s) < j_{i+1} \ \forall\ i_s\notin\{\alpha_1,\ldots,\alpha_i\}$.\\
(2) $v(j_{i+1})=i_{s_0} (for\ some\ s_0\ such\ that\ p\geq s_0 > 1)\in \{\alpha_1, \dots, \alpha_i\}$ and $v^{-1}(i_s) < j_{i+1} \ \forall\ i_s\in\{i_1,\ldots,i_{s_0-1}\}\setminus\{\alpha_1,\ldots,\alpha_i\}$\\ or $v(j_{i+1})=i_1 \in \{\alpha_1, \dots, \alpha_i\}$ .\\
(3) $i_{s_0} < v(j_{i+1}) < i_{{s_0 + 1}}$, for some $s_0 \in \{1,\dots,p-1\}$ and $v^{-1}(i_s) < j_{i+1} \ \forall\ i_s\in\{i_1,\ldots,i_{s_0}\}\setminus\{\alpha_1,\ldots,\alpha_i\}$. \par 
Let us call the complement of the above conditions as \textbf{conditions $\delta_i $ ($1\leq i \leq p-1$)}. Then, we have the following proposition : 
\begin{prop}\label{p.delta_i}
	Assume that the conditions of Lemma \ref{l.LD} doesn't hold. Then, there exist a non zero path in $A$ if and only if there exist a non zero entry in the $j_1^{th}$ column, (say at the $\alpha_1^{th}$ row), such that the $\delta_i$ holds for each $i$, $1\leq i \leq p-1$. 
\end{prop}
\begin{example}\label{e.delta_i}
	Let us borrow the same $S_n,v,w$, as in Example \ref{e.discardingminors} and look at the southwest sub-matrix $Z_{33}^{(v)}$. We have:
	\[
	Z_{33}^{(v)}=\begin{bmatrix}
		z_{31}&1&0\\
		z_{21}&z_{22}&1\\
		z_{11}&z_{12}&z_{13}
	\end{bmatrix}\]	
	Here one sees that $p=3, i_k=k=j_k$ for $1\leq k\leq 3$ w.r.t $Z^{(v)}$. Pick any entry from first column, say it is in $\alpha_1^{th}$ row. For each $i\in\{1,2\}$, consider the conditions:\\
	(1) $v(i+1) > 3$ and $v^{-1}(s) < {i+1} \ \forall\ s\in\{1,2,3\}\setminus\{\alpha_1,\ldots,\alpha_i\}$.\\
	(2) $v(i+1)={s_0} \ (3\geq s_0 > 1)\in \{\alpha_1, \dots, \alpha_i\}$ and $v^{-1}(s) < {i+1} \ \forall\ s\in\{1,\ldots,s_0-1\}\setminus\{\alpha_1,\ldots,\alpha_i\}$ or $v(i+1)=1 \in \{\alpha_1, \dots, \alpha_i\}$ .\\
	(3) ${s_0} < v({i+1}) < {{s_0 + 1}}$, for some $s_0 \in \{1,2\}$ and $v^{-1}(s) < {i+1} \ \forall\ s\in\{1,\ldots,s_0\}\setminus\{\alpha_1,\ldots,\alpha_i\}$. \\
	Then, $\delta_i$ is the complement of the statement that ``at least one of the conditions (1),(2) and (3) should hold for each i ($1\leq i \leq 2$)". We see that no matter which entry we pick from the first column of $Z_{33}^{(v)}$, for each $i$, none of the conditions (1),(2) or (3) hold, i.e, for each i, $\delta_i$ holds. Hence, there exist a non zero path in A, demonstrating the above Proposition.
\end{example}
For computational convenience, we now provide certain conditions on $v$ (which we call \textbf{conditions} $\gamma_i$), which tell us when does there exist a nonzero path in a minor $A$ of $Z^{(v)}$. We will now state the conditions $\gamma_i$ after some preparation, keeping in mind the ``\textbf{change in convention}'' mentioned earlier.\par 

Let $A$ be any $p \times p$ ($p \geq 2$) minor of $Z^{(v)}$. Any such $A$ that doesn't have a column of the form $(0,\dots,0,1,\dots,z,\dots)^t$ is already homogeneous, by Theorem \ref{t.hompaththm}. Then (for the purpose of finding inhomogeneous minors), by using the same theorem, we need to focus on the minors $A$ that do have a column of this form (namely, of the form $(0,\ldots,0,1,\ldots,z,\ldots)^t$), and a non zero path through an indeterminate $z$ lying below 1.
\begin{remark}\label{r.reducingpaths}
Let $A$ be as in the previous paragraph. Notice that a non zero path through $z$ (lying in a column of the form $(0,\ldots,0,1,\ldots,z,\ldots)^t$) must pick an indeterminate, say $z'$ from the row of 1, from its left. So, we can remove the row of 1 and the column where $z'$ is from, and focus on finding a nonzero path through some $z$ in the $(p-1)\times (p-1)$ sub-matrix thus obtained (which has now the column of $z$ of the form $(0,\dots,0,\dots,z,\dots)^t$), because we can always add the $z'$ to the path obtained in the sub-matrix in appropriate position and get a non zero path through $z$ in $A$.
\end{remark}
Abusing notation, let us call the sub-matrix also as $A$ and assume WLOG that $A$ is of size $p$, with rows $\{i_1 < \dots < i_p\}$ and columns $\{j_1 < \dots < j_p\}$ in relation to $Z^{(v)}$ (the rows being counted from bottom to top and the columns being counted from left to right, in increasing order). Let  $(0,\dots,0,\dots,z,\dots)^t$ be the column $j_{d}$ and pick an indeterminate $z$, say, from row $i_c$.\\
	 Now, notice that if $d=1$, the other columns of the matrix are $j_2,\dots,j_p$ whereas if $1<d<p$, then the other columns are $j_1,\dots , j_{d-1},j_{d+1},\dots,j_p$. So, we make a convention here, that, while chosen $z$ in $(i_c,j_d)$ entry of $A$, the remaining columns we label as $j_1,\dots,j_{d-1},j_{d+1},\dots,j_p$ from left to right. i.e, for $d=1$. $j_i=j_{i-1}$ for $1<i\leq d$ and $j_i$ remains same for $i>d$.  \par 
Consider the first column $j_1$. We must have that there exists at least one non-zero entry in a row other than $i_c$. The condition for not having a non zero entry in column $j_1$ in a row other than $i_c$ is that at least one of the following two conditions must hold:\\
(1) $v(j_1) > i_p$ and $v^{-1}(i_s) < j_1 \ \forall\ s\in\{1,\ldots,p\}\setminus\{c\}$ or $v(j_1)<i_1$.\\
(2) $i_{s_0} < v(j_1) < i_{{s_0 + 1}}$, for some $s_0 \in \{1,\dots,p-1\}$ and $v^{-1}(i_s) < j_1 \ \forall\ s\in\{1,\ldots,s_0\}\setminus\{c\}$.\\
(3) $v(j_1)=i_c$  and $v^{-1}(i_s) < j_{1} \ \forall\ s\in\{1,\ldots,c-1\}$ (notice that this condition makes sense only if $d=1$)\\
We call the complement of the above \textbf{condition as $\gamma_0$}, i.e, we say $\gamma_0$ holds if none of the conditions above, hold.\\
So, assuming $\gamma_0$ holds, it must be that we have a non zero entry at the $j_1^{th}$ column in some row other than $i_c$. Pick one such entry. Call it $z_{{i_{\alpha_1}},j_1}$. In order for us to have a non zero path through $z_{{i_{\alpha_1}},j_1}$, we must have a non-zero entry at column $j_2$ in any row other than $i_{\alpha_1}$ or $i_{c}$. Now, the condition for not having a non zero entry at the $j_2^{th}$ column (except maybe at the $i_{\alpha_1}^{th}$ and the $i_c^{th}$ row) is that at least one of the following conditions should hold:\\
(1) $v(j_2) > i_p$ and $v^{-1}(i_s) < j_2 \ \forall\ s\in\{1,\ldots,p\}\setminus\{\alpha_1,c\}$ or $v(j_2)<i_1$.\\
(2) $i_{s_0} < v(j_2) < i_{{s_0 + 1}}$, for some $s_0 \in \{1,\dots,p-1\}$ and $v^{-1}(i_s) < j_2 \ \forall\ s\in\{1,\ldots,s_0\}\setminus\{\alpha_1,c\}$. \\
(3) $v(j_{2})=i_{s_0} (p\geq s_0 > 1), s_0 \in\{\alpha_1,c\}$ and $v^{-1}(i_s) < j_{i+1} \ \forall\ s\in\{1,\ldots,s_0-1\}\setminus\{\alpha_1,c\}$.\par 

Let $i\in\{1,\ldots,p-1\}$ be arbitrary. For a general $j_{i+1}^{th}$ column (excluding column $j_d$), the condition for not having a non zero entry at the $j_{i+1}^{th}$ column (except maybe at the $i_{\alpha_k}^{th}$ and $i_c^{th}$ row, where $1\leq k \leq i$) is that at least one of the following conditions should hold:\\
For $i\in\{1,\dots,p-1\}\setminus \{d-1\}$\\
(1) $v(j_{i+1}) > i_p$ and $v^{-1}(i_s) < j_{i+1} \ \forall\ s\in\{1,\ldots,p\}\setminus\{\alpha_1,\ldots,\alpha_i,c\}$.\\
(2) $i_{s_0} < v(j_{i+1}) < i_{{s_0 + 1}}$, for some $s_0 \in \{1,\dots,p-1\}$ and $v^{-1}(i_s) < j_{i+1} \ \forall\ s\in\{1,\ldots,s_0\}\setminus\{\alpha_1,\ldots,\alpha_i,c\}$. \\
(3) $v(j_{i+1})=i_{s_0} (p\geq s_0 > 1), s_0 \in \{\alpha_1, \dots, \alpha_i,c\}$ and $v^{-1}(i_s) < j_{i+1} \ \forall\ s\in\{1,\ldots,s_0-1\}\setminus\{\alpha_1,\ldots,\alpha_i,c\}$
\par
Let us call the complement of the above conditions as \textbf{conditions $\gamma_i $($1\leq i \leq p-1$)}, i.e , for each $i$, we say $\gamma_i$ holds, if none of the conditions $(1),(2),(3)$ hold. Recall that we also have defined \textbf{condition $\gamma_0$} above.\\
Together, we have the following lemma:
\begin{lemma}\label{l.gamma_i}
Let A be any $p \times p$ ($p \geq 2$) in relation to $Z^{(v)}$, having a column of the form $(0,\dots,0,\dots,z,\dots)^t$. Then, there exists a non zero path through any point $z$ of this column if and only if the conditions $\gamma_i$ above, hold for each $i$, $0\leq i \leq p-1$.
\end{lemma}
\begin{example}\label{e.gamma_i}
	Let us take the same $S_n,v,w$ as in Example \ref{e.discardingminors} and look at the southwest sub-matrix $Z_{33}^{(v)}$. We have:
	\[
	Z_{33}^{(v)}=\begin{bmatrix}
		z_{31}&1&0\\
		z_{21}&z_{22}&1\\
		z_{11}&z_{12}&z_{13}
	\end{bmatrix}\]	
	Here one sees that $p=3, i_k=k=j_k$ for $1\leq k\leq 3$ w.r.t $Z^{(v)}$. Suppose we are to check whether there is a non zero path through $z_{13}$. Then, following the discussion above the lemma, it is equivalent to checking whether there is a non zero path through $z_{13}$ in $\begin{bmatrix}
		z_{31}&0\\
		z_{11}&z_{13}
	\end{bmatrix}$ or $\begin{bmatrix}
		{1}&0\\
		z_{12}&z_{13}
	\end{bmatrix}$. Similar to how we did in Example \ref{e.delta_i}, we write down the conditions (1),(2),(3) for a general $i$ and taking $\gamma_i$ as its complement, we see that for each i, $\gamma_i$ holds, i.e, there is a non zero path through $z_{13}$ in the $2\times2$ sub-matrices of $Z_{33}^{(v)}$ which then gives a non zero path through $z_{13}$ in $Z_{33}^{(v)}$. 
\end{example}

\section{The algorithm for testing inhomogeneity of KL-ideals}\label{s.whichwtodiscard}
In this section, we provide an algorithm for testing a KL-ideal for inhomogeneity.
\begin{itemize}
 \item Discard $w=w_0=n\ n-1\ n-2\ \cdots\ 1$ (See Lemma \ref{l.nodefminors} in this regard).
 \item Discard all pairs $(v,w)$ for which $\Tilde{R}_v\nleq\Tilde{R}_w$ (See Theorem \ref{t.wholeRing} in this regard).
 \item Fix any pair $(v,w)$ which is not yet discarded. 
 \item Use \S \ref{ss.rmatrix} to obtain a formula for $\Tilde{R}_w$.
 \item Then use the concept of an essential set from \S \ref{ss.redminors} to discard the unnecessary defining minors of the KL-ideal $I_{v,w}$.
 \item Next, look at the structure of the remaining defining minors of $I_{v,w}$. Use Lemmas \ref{l.zeroRorC}, \ref{l.LD} and Proposition \ref{p.delta_i} to discard those defining minors, which give singular determinant. While using the results Lemma \ref{l.LD} and Proposition \ref{p.delta_i}, we should keep in mind the change in convention, which is mentioned just before Lemma \ref{l.LD}.
 \item Then apply Theorem \ref{t.hompaththm} to get hold of all inhomogeneous defining minors, which are not yet discarded. In doing so, we have to use Lemma \ref{l.gamma_i} and Remark \ref{r.reducingpaths}. While using Lemma \ref{l.gamma_i} and Remark \ref{r.reducingpaths} also, we should keep in mind the change in convention, which is mentioned just before Lemma \ref{l.LD}.\par 
 Now, we have a set of generators (consisting of certain defining minors) of the KL-ideal $I_{v,w}$, out of which some generators are inhomogeneous, and some are homogeneous.
 \item We declare the KL-ideal $I_{v,w}$ to be \textbf{inhomogeneous} if it satisfies the sufficient condition(s) given by the contrapositive statement of Lemma \ref{l.homthm} (taking into account Remark \ref{r.lemma-homthm}). To do this checking, we take help from the theorems \ref{t.poldiv1},\ref{t.poldiv2}, and \ref{t.poldiv3}.
\end{itemize}
\section{Example on the execution of the algorithm}\label{s.neye}
In this section, we provide an example which executes the algorithm given in \S \ref{s.whichwtodiscard}.
\begin{example}\label{e.Neye2}
Consider $S_n$ for $n\geq 6$.\\
	 Let $w\in S_n$ be such that  $w(1)=n, w(2)=n-2, w(3)=n-1, w(4)=n-4, w(5)=n-3$ and $w(i)=n-i+1 \ \forall \ 6\leq i\leq n$.\\
	 Also, let $v \in S_n$ be such that $v(1)=n-3,v(2)=n-1,v(3)=n-2,v(4)=n-4$ and for the value of $v$ in remaining indices, one can take any permutation of the letters $\{1,2,\dots,n-5,n\}$ such that $\tilde{R}_v \leq \tilde{R}_w$ in the remaining $(n-4)$ columns of the matrix $\tilde{R}_v$. Notice that one can simply take $v(5)=n$ and $v(i)=w(i)$ for $6\leq i\leq n$ as one such permutation and this ensures that such permutations do exist.\\
	 $\bullet$ Now, according to the algorithm above, we start with finding the rank matrix, in accordance with point 4: \\
	 \[
	 \tilde{R}_w =
	 \begin{bmatrix}
	 	1 & 2 &3 &4 &5 & 6 & \dots & n\\
	 	\vdots & \vdots & \vdots  & \vdots & \vdots & \vdots & \dots & \vdots\\
	 	1 & 2 &3 &4 &5 & 6 & \dots & 6\\
	 	1 & 2 &3 &4 &5 & 5 & \dots & 5\\
	 	1 & 2 &3 &3 &4 & 4 & \dots & 4\\
	 	1 & 2 &3 &3 &3 & 3 & \dots & 3\\
	 	1 & 1 &2 &2 &2 & 2 & \dots & 2\\
	 	1 & 1 &1 &1 &1 & 1 & \dots & 1\\
	 \end{bmatrix}
	 \]
	 Observe that irrespective of any $v\in S_n$, there will be no minors $Z_{st}^{(v)}$ for $t\geq 5$. This is because for $t\geq 5$ the column entries of $\tilde{R}_w$ look like :
	 \[
	 \tilde{r}_{st}^{(w)}=\begin{cases}
	 	t &\textnormal{for}\quad 1\leq s \leq n-t+1\\
	 	n-s+1 &\textnormal{for} \quad n-t+2 \leq s \leq n
	 \end{cases}
	 \]
	 and this clearly shows that $\tilde{r}_{st}^{(w)} > min\{n-s+1,t\}$ for $t\geq 5$ and so from Section \ref{ss.defminors} the observation follows.
	 \\~\\
	 
	 \par
	 Now by doing so, one can check then we have $\tilde{R}_v \leq \tilde{R}_w$ altogether and as such, by Theorem \ref{t.wholeRing}, $I_{v.w}$ is not the whole ring. 
	For instance, for $v(i)$ as defined above, and the particular case mentioned where $v(5)=n$ and $v(i)=w(i)$ for $6\leq i\leq n$, we have :
	\[
	\tilde{R}_v=\begin{bmatrix}
		1&2&3&4&5&\dots&n\\
		\vdots&\vdots&\vdots&\vdots&\vdots&\vdots\\
		1&2&3&4&5&\dots&5\\
		1&2&3&3&4&\dots&4\\
		0&1&2&2&3&\dots&3\\
		0&1&1&1&2&\dots&2\\
		0&0&0&0&1&\dots&1
	\end{bmatrix}
	\]
	and from this, one can readily see  $\tilde{R}_v \leq \tilde{R}_w$.\par
	So, this takes care of the first two points of the algorithm and says that we are in point 3.\\
	$\bullet$ Now, we continue with point 5 and 6 of the algorithm and find that the pairs corresponding to the essential set from \ref{ss.redminors} are $(n-1,2), (n-3,4)$ and the only relevant minors for $Z^{(v)}$ come from the submatrices $Z_{n-1,2}^{(v)}$ and $Z_{n-3,4}^{(v)}$ which are :\\
	\[
	\begin{bmatrix}
		z_{21}&1\\
		z_{11}&z_{12}
	\end{bmatrix}
	\quad \textnormal{and} \quad
	\begin{bmatrix}
	1&0&0&0\\
	z_{31}&0&1&0\\
	z_{21}&1&0&0\\
	z_{11}&z_{12}&z_{13}&z_{14}
	\end{bmatrix}
	\quad \textnormal{respectively.}
	\]
$\bullet$ Now, we go to point 7 and use the necessary results there (in this case Theorem \ref{t.hompaththm})  to see that $f_1=z_{21}z_{12}-z_{11}$ is inhomogeneous and $f_2=-z_{14}$ is homogeneous directly.\\
$\bullet$ Hence $I_{v,w}=\langle f_1,f_2 \rangle$ is inhomogeneous as follows from Lemma \ref{l.homthm} which is our point 8 of the algorithm.\\
Notice that the first minor will give an inhomogeneous polynomial also follows directly from Theorem \ref{t.hompaththm}.
\vspace{0.1in}\\
In fact, for each $n\geq 6$, the above example provides corresponding $v,w \in S_n$ for which $I_{v,w}$ is inhomogeneous.
\end{example}

\section{Polynomial mutation}\label{s.poly-mutation}
In this section, we introduce a process (which we call the ``Polynomial mutation''), which helps us to obtain certain conditions for the identification of homogeneous KL-ideals. The process ``polynomial mutation'' is a general process involving finitely many polynomials in $\mathbb{C}[x_1\dots,x_d]$, which when applied to the case of KL-ideals helps us in obtaining conditions for their homogeneity. This process ``polynomial mutation'' is based purely on polynomial division.
\subsection{Description of the polynomial mutation process}\label{ss.poly-mut-description}
Let $d$ be any positive integer, and $m$ be any positive integer $\leq d$. Let $f_1, f_2,\dots, f_m$ be finitely many polynomials in $\mathbb{C}[x_1\dots,x_d]$ and
			\begin{center}
				 $f_i=\sum_{k=1}^{n_i}X_{ik}$ for each $1\leq i\leq m$,
				\end{center}
				 where $X_{ik}$ denote the terms of $f_i$ including constant coefficients. The first subscript `i' in $X_{ik}$ denotes that $X_{ik}$ has come from $f_i$. Note that there is no ordering among the terms of any $f_i$.
				 \vspace{5pt}\\
				 Now, let $n$ denote any non negative integer.
	\begin{defi}\label{d.stageNterms}
		An expression of the form $\frac{X_{r_{0}s_{0}}\cdot X_{r_{1}s_{1}^{'}}\cdot X_{r_{2}s_{2}^{'}}\dots X_{r_{n}s_{n}^{'}}}{X_{r_{1}s_{1}}\cdot X_{r_{2}s_{2}}\cdot X_{r_{3}s_{3}}\dots X_{r_{n+1}s_{n+1}}}$ where $s_j \neq s_j^{'} \ \forall \ j\in\{1,2\dots,n\}$ and for each $j\in\{1,2\dots,n\}$, the denominator of $\frac{X_{r_{0}s_{0}}\cdot X_{r_{1}s_{1}^{'}}\cdot X_{r_{2}s_{2}^{'}}\dots X_{r_{j-1}s_{j-1}^{'}}}{X_{r_{1}s_{1}}\cdot X_{r_{2}s_{2}}\cdot X_{r_{3}s_{3}}\dots X_{r_{j}s_{j}}}$ divides the numerator, is called a ``stage n monomial of $\{f_1,\dots,f_m\}$".
	\end{defi}
	\begin{defi}\label{jdivisors}
		Given any stage $n$ monomial $\frac{X_{r_{0}s_{0}}\cdot X_{r_{1}s_{1}^{'}}\cdot X_{r_{2}s_{2}^{'}}\dots X_{r_{n}s_{n}^{'}}}{X_{r_{1}s_{1}}\cdot X_{r_{2}s_{2}}\cdot X_{r_{3}s_{3}}\dots X_{r_{n+1}s_{n+1}}}$ of $\{f_1,\dots,f_m\}$, we call $X_{r_js_j}$ for each $j\in\{1,2\dots,n\}$ as the ``$j^{th}$ divisor of stage $n$".
	\end{defi}
	\begin{defi}\label{generatedterm}
		An expression of the form $\frac{X_{r_{0}s_{0}}\cdot X_{r_{1}s_{1}^{'}}\cdot X_{r_{2}s_{2}^{'}}\dots X_{r_{n}s_{n}^{'}}}{X_{r_{1}s_{1}}\cdot X_{r_{2}s_{2}}\cdot X_{r_{3}s_{3}}\dots X_{r_{n+1}s_{n+1}}}\cdot X_{r_{n+1}s_{n+1}^{'}}$ where $s_j \neq s_j^{'} \ \forall \ j\in\{1,2\dots,n+1\}$ and the portion $\frac{X_{r_{0}s_{0}}\cdot X_{r_{1}s_{1}^{'}}\cdot X_{r_{2}s_{2}^{'}}\dots X_{r_{n}s_{n}^{'}}}{X_{r_{1}s_{1}}\cdot X_{r_{2}s_{2}}\cdot X_{r_{3}s_{3}}\dots X_{r_{n+1}s_{n+1}}}$ is a stage $n$ monomial of $\{f_1,\dots,f_m\}$ , is called a ``stage n generated term of $\{f_1,\dots,f_m\}$" and the term $X_{r_{n+1}s_{n+1}^{'}}$ is called an ``$(n+1)^{th}$ multiplier of stage $n$".
	\end{defi}
Note that in $X_{r_js_j}$ and $X_{r_Js'_j}$ above, the $r_j$'s are the same, only the $s_j\neq s'_j$. And this is true $\forall j\in \{1,2\dots,n+1\}$.
\begin{defi}\label{d.genterm-originating}
A stage $n$ generated term of $\{f_1,\ldots,f_m\}$ of the form 
$$\frac{X_{r_{0}s_{0}}\cdot X_{r_{1}s_{1}^{'}}\cdot X_{r_{2}s_{2}^{'}}\dots X_{r_{n}s_{n}^{'}}}{X_{r_{1}s_{1}}\cdot X_{r_{2}s_{2}}\cdot X_{r_{3}s_{3}}\dots X_{r_{n+1}s_{n+1}}}\cdot X_{r_{n+1}s_{n+1}^{'}}$$ is said to be ``originating from $X_{r_0s_0}$''. Similarly, for a stage $n$ monomial of $\{f_1,\ldots, f_m\}$.
\end{defi}
\begin{defi}\label{d.part-of-fi}
	Fix $i\in \{1,2\dots,m\}$ arbitrarily. By a part $\tilde{f_i}$ of $f_i$, we mean a sum of some terms (if not all) of $f_i$.
\end{defi}
	Let $\tilde{f_i}=X_{ij_1}+X_{ij_2}+\dots + X_{ij_p}$ be a part of $f_i$. Note that there always exists terms of $f_1,\dots,f_m$ such that they divide the terms of $\tilde{f_i}$(the terms of $\tilde{f_i}$ itself are examples of such terms). In fact, there may exist more than one term(s) of $f_1\dots,f_m$ which divide any given term of $\tilde{f_i}$.\\ \\
The ``\textbf{polynomial mutation process}'' then starts from step $0$ as follows:\\	
\noindent\textbf{Step 0:} Pick some (possibly all) terms $X_{ij_{l_1}}, X_{ij_{l_2}}, \dots, X_{ij_{l_t}}$ from the terms of $\tilde{f_i}$. Then, pick one term (exactly one) from the $f_1, f_2, \dots , f_m$ which divides $X_{ij_{l_k}}$ for each  $k\in \{1,\dots,t\}$.\\
(Note that the term picked from $f_1, \dots, f_m$ which divides $X_{ij_{l_k}}$ may or may not equal $X_{ij_{l_k}}$ itself. Also, the picked terms for distinct $X_{ij_{l_k}}$ need not be distinct.)\\
Say $X_{r_ks_k}$ is the term from $f_1,\dots, f_m$ which divides $X_ij_{l_k}$ ($\forall k \in \{1,\dots,t\}$).\\
To keep track of the division, we denote $X_{r_ks_k}$ by $X^{j_{l_k}}_{r_ks_k}$ to imply that $X_{r_ks_k}$ divides $X_{ij_{l_k}}$.\\
We now construct a finite sequence  $\tilde{g}^{(0)}_{1}, \dots , \tilde{g}^{(0)}_{m}$ of polynomials in $\mathbb{C}[x_1,\dots,x_d]$ in the following manner:\\
For any $j\in \{1,\dots,m\}$, define\\
\begin{equation}
	\tilde{g}^{(0)}_{j} := \begin{cases}
		0 & \text{ if } \{X^{j_{l_1}}_{r_1s_1}
		,\dots, X^{j_{l_t}}_{r_ts_t}\}\cap \{\text{the terms of }  f_j\}=\emptyset \\
		\sum_{b=1}^{p}\frac{X_{ij_{l_{k_b}}}}{X^{j_{l_{k_b}}}_{r_{k_b}s_{k_b}}} & \text{ if } \{X^{j_{l_1}}_{r_1s_1}
		,\dots, X^{j_{l_t}}_{r_ts_t}\}\cap \{\text{the terms of }  f_j\} = \{X^{j_{l_{k_1}}}_{r_{k_1}s_{k_1}}, \dots, X^{j_{l_{k_p}}}_{r_{k_p}s_{k_p}} \} \\
	\end{cases}
\end{equation}
Observe that all the non zero terms of each $\tilde{g}^{(0)}_{j}(1\leq j \leq m)$ are nothing but some stage $0$ monomials of $\{f_1, \dots, f_m\}$.\\
Then, we obtain
\begin{equation} \label{e1}
	\sum_{j=1}^{m}\tilde{g}^{(0)}_{j}f_j=\tilde{f}^{(0)}_{i} + \text{some (possibly none) stage 0 generated terms of } \{f_1 \dots, f_m\}.
\end{equation}
where $\tilde{f}^{(0)}_{i}$ is the part of $\tilde{f}_{i}$ consisting of $X_{ij_{l_1}},\dots , X_{ij_{l_t}}$ only. We denote equation \ref{e1} above by \textbf{Equation ($*0$)}.
\\ \\
Now, $4$ cases arise:\\
\textbf{Case 1:}\label{case1} Equation \ref{e1} is of the form $	\sum_{j=1}^{m}\tilde{g}^{(0)}_{j}f_j=\tilde{f}^{(0)}_{i}$\\
If this happens for some choice of terms at stage $0$, we say that ``\textbf{polynomial mutation terminates successfully after stage $0$ w.r.t $\tilde{f}_i$ and $\{f_1,\dots, f_m\}$}".\\
\textbf{Case 2:}\label{case2} If for no choice of terms at stage $0$, we can be in \hyperref[case1]{Case 1}, and equation \ref{e1} is of the form 
$$\sum_{j=1}^{m}\tilde{g}^{(0)}_{j}f_j=\tilde{f}^{(0)}_{i} +  \text{some stage 0 generated terms of } \{f_1 \dots, f_m\},$$
where the part on the RHS of the above equation (excluding $\tilde{f}^{(0)}_{i}$) is non trivial which is such that \\
$\exists$ at least one term in it which is not a term of $\tilde{f}_i$ and for which there $\nexists$ any term in $f_1,\dots,f_m$( other than the $(1)$-th multiplier of stage $0$ present inside it) dividing it.\\
If this happens, we say that ``\textbf{Polynomial mutation stops abruptly after stage 0 w.r.t $\tilde{f}_i$ and $f_1,\dots,f_m$}".\\
\textbf{Case 3:}\label{case3} If for no choice of terms at stage $0$, we can be in \hyperref[case1]{Case 1} and $\forall$ choices of terms at stage $0$, equation \ref{e1} is of the form $	\sum_{j=1}^{m}\tilde{g}^{(0)}_{j}f_j $= $
\tilde{\tilde{f}}^{(0)}_{i} $ where $\tilde{\tilde{f}}^{(0)}_{i}\neq \tilde{f}_{i} $, and $\tilde{\tilde{f}}^{(0)}_{i}$ contains all terms of $\tilde{f}^{(0)}_{i}$ plus or minus some terms of $\tilde{f}_{i}$. (Note that it is also possible to have $\tilde{\tilde{f}}_{i}^{(0)}$ equal to $0$.)\\
If this happens, then also we say that ``\textbf{polynomial mutation stops abruptly after stage $0$ w.r.t $ \tilde{f}_i$ and $\{f_1,\dots, f_m\}$}".\\
\textbf{Case 4: }\label{case4} If we are neither in \hyperref[case1]{Case 1}, nor in \hyperref[case2]{Case 2} or \hyperref[case3]{Case 3}.\\
Then, for some choices of terms at stage $0$, equation \ref{e1} is of the form $$ \sum_{j=1}^{m}\tilde{g}^{(0)}_{j}f_j=\tilde{f}^{(0)}_{i} +  \text{some stage 0 generated terms of } \{f_1 \dots, f_m\},$$
where the part on RHS on the above equation (excluding $\tilde{f}^{(0)}_{i}$) is non trivial which is such that:
\begin{quote}
For each term in it(which is not a term of $\tilde{f}_{i}$),$\exists$ at least one term in $f_1,\dots, f_m$ (other than the ($1$)-th multiplier of stage $0$ present inside it) which is a divisor.
\end{quote}
If this happens, then we proceed as follows to the next step:\\
\noindent\textbf{Step 1:} Pick some(possibly all) terms $X^{(0)}_{a_1b_1}, \dots , X^{(0)}_{a_{r_1}b_{r_1}}$ from the part on the RHS of equation \ref{e1} (excluding terms of $\tilde{f}_i$). Note that these $X^{(0)}_{a_eb_e}(1\leq e \leq r_1)$ are not in general terms of $f_{a_e}$. (Always consider $X^{(0)}_{a_kb_k}$ terms together with the scalar coefficients.)\\ \\
Then for each $e\in\{1,\dots,r_1\}$, pick one term (exactly one) from $f_1,\dots,f_m$ (other than the ($1$)-th multiplier of stage $0$ present inside $X^{(0)}_{a_eb_e}$) which divides $X^{(0)}_{a_eb_e}$.\par
Say $X_{u_ev_e}$ is the term from $f_1,\dots,f_m$ which divides $X^{(0)}_{a_eb_e}(\forall e\in\{1,\dots,r_1\})
$ which is not equal to the ($1$)-th multiplier of stage $0$, present inside $X^{(0)}_{a_eb_e}$.\\
To keep track of the division, we denote $X^{(0)}_{u_ev_e}$ by $X^{(a_eb_e)^{0}}_{u_ev_e}$ to imply that $X_{u_ev_e}$ divides $X^{(0)}_{a_eb_e}$. \\
We now construct a finite sequence $\tilde{g}^{(1)}_1,\dots, \tilde{g}^{(1)}_m$ of  polynomials in $\mathbb{C}[x_1,\dots , x_d]$ in the following manner:\\
For any $j\in\{1,\dots,m\}$, define:
\[
\tilde{g}^{(1)}_j :=\begin{cases}
	\tilde{g}^{(0)}_j & \text{ if } \{X^{(a_1b_1)^{0}}_{u_1v_1},\dots, X^{(a_{r_1}b_{r_1})^{0}}_{u_{r_1}v_{r_1}}\} \cap \{\textnormal{the terms of } f_j\}=\emptyset\\
	\tilde{g}^{(0)}_j -\sum_{s=1}^{q}\frac{X^{(0)}_{a_{\alpha_s}b_{\alpha_s}}}{X^{(a_{\alpha_s}b_{\alpha_s})^{0}}_{u_{\alpha_s}v_{\alpha_s}}}& \text{ if } \{X^{(a_1b_1)^{0}}_{u_1v_1},\dots, X^{(a_{r_1}b_{r_1})^{0}}_{u_{r_1}v_{r_1}}\} \cap \{\textnormal{the terms of } f_j\}=\\
	&\{X^{(a_{\alpha_1}b_{\alpha_1})^{0}}_{u_{\alpha_1}v_{\alpha_1}},\dots, X^{(a_{\alpha_q}b_{\alpha_q})^{0}}_{u_{\alpha_q}v_{\alpha_q}} \}
\end{cases}
\]
Observe that all the nonzero terms of each $\tilde{g}^{(1)}_j(1\leq j\leq m)$ are either some stage 0 or stage 1 monomials of $\{f_1,\dots,f_m\}$.\\ \\
If, in this process of constructing $\tilde{g}^{(1)}_j$ from $\tilde{g}^{(0)}_j$, some term(s) of $\tilde{g}^{(0)}_j$ \textit{get canceled} (for all choices of $X^{(0)}_{a_ib_i}$ and $X^{{(a_ib_i)}^{0}}_{u_iv_i}$, then we also say that ``\textbf{Polynomial mutation stops abruptly due to cancellation after stage 0 w.r.t $\tilde{f}_i$ and $\{f_1,\dots,f_m\}$}". This situation causes repetition of terms in the resulting expression $\sum_{j=1}^{m}\tilde{g}^{(1)}_jf_j$.\\ \\
If not, then we obtain,\\
\begin{equation}\label{mut.e2}
	\begin{split}
		\sum_{j=1}^{m}\tilde{g}^{(1)}_jf_j=\tilde{f}^{(1)}_i+ \textnormal{some (possibly none) stage 0 generated terms of} \{f_1,\dots,f_m\}\\
		+ \textnormal{some(possibly none) stage 1 generated terms of} \{f_1,\dots,f_m\},
	\end{split}
	\raisetag{-20pt}
\end{equation}
where $\tilde{f}^{(1)}_i$ is either a part of $\tilde{f}^{(0)}_i$ or a part of $\tilde{f}_i$ which properly contains $\tilde{f}^{(0)}_i$. We denote equation \ref{mut.e2} above by \textbf{Equation ($*1$)}.\\ \\
Again, four cases similar to stage 0 arise and we proceed similarly to \textbf{Step 2}. In step 2 also, we subtract terms from $\tilde{g}_j^{(1)}$ to construct the new polynomials $\tilde{g}_j^{(2)}$ (In fact, we do subtraction only in each step).\\ \\
In step 2, if mutation does not stop abruptly (due to any reason), then we pick terms which are either stage 0 or stage 1 generated terms (or both) from equation \ref{mut.e2} which have a divisor (other than the (1)-th multiplier of stage 0 or the (2)-th multiplier of stage 1, present inside it).
\begin{remark}\label{mut.remark1}
	Note that if ``Polynomial mutation stops abruptly due to cancellation after stage 0 w.r.t $\tilde{f}_i$ and $\{f_1,\dots,f_m\}$, then we must have the following situation occurring once or more :\\
	A term of $	\tilde{g}^{(0)}_j=-\frac{X_{a_{\alpha_s}b_{\alpha_s}}^{(0)}}{X^{(a_{\alpha_s}b_{\alpha_s})^{0}}_{u_{\alpha_s}v_{\alpha_s}}}$\\
	for some $s\in \{1,\dots, q\}$ and some $j\in \{1,\dots,m\}$.
\end{remark}

If this process of polynomial mutation terminates successfully after stage $N$ (for some non negative integer $N$) w.r.t $\tilde{f}_i$ and $\{f_1,\dots,f_m\}$ for some choices of terms at each of the $0^{th},1^{st},\dots, (N-1)^{th}$ stages, then we say that ``\textbf{Polynomial mutation terminates successfully w.r.t $\tilde{f}_i$ and $\{f_1,\dots,f_m\}$}" and we write :
$$\tilde{f}_i=\sum_{j=1}^{m}\tilde{g}^{(N)}_jf_j.$$
This finishes the description of the process \textbf{Polynomial mutation}.
\begin{theorem}\label{t.iffhomthm}
	Let $I=\langle f_1,\dots,f_m \rangle$. Then a homogeneous component $\tilde{f_i}$ of a $f_i$ belongs to the ideal $I$ if and only if there exists polynomials $g_j$ ($1\leq j\leq m$) such that $\tilde{f_i}=\Sigma_{j=1}^{m}g_jf_j$, where each term of each $g_j$ is of the form of a stage $p$ monomial (for some integer $p\geq 0$).
\end{theorem}
\begin{proof}
	$\Longleftarrow$ obvious.\\
	$\Longrightarrow$ Choose a $f_i$ ($1\leq i\leq m$). Let $\tilde{f_i}$ be a homogeneous component of $f_i$. Since $I$ is homogeneous, we must have that $\tilde{f_i}\in I$. This implies that
	$$\tilde{f_i}=\Sigma_{j=1}^{m}g_jf_j$$
	for some polynomials $g_j$. \\
	\noindent\textbf{Claim:} Without loss of generality, we can assume that the terms of all the $g_j$'s in the above expression are of the required form. \\
	\noindent\textbf{Proof of the claim:} Clearly, the $g_j$'s contain some terms which are stage-$0$-monomials of $\{f_1,\dots,f_m\}$ (originating from a term of $\tilde{f_i}$), since otherwise, the sum $\Sigma_{j=1}^{m}g_jf_j$ can't yield a term of $\tilde{f_i}$.
	Suppose now that $X$ is a term of some $g_k$ ($1\leq k\leq m$), which is not a stage-$p$-monomial for any integer $p\geq 0$.\\
	Then $2$ cases arise:\\
	\noindent Case 1: $X.X_{ks}=$ some stage $l$ generated term.\\
	Say, $X.X_{ks}=\frac{X_{r_0s_0'}.X_{r_1s_1'}.....X_{r_ls_l'}}{X_{r_1s_1}.X_{r_2s_2}.....X_{r_{l+1}s_{l+1}}}X_{r_{l+1}s_{l+1}'}$. Then,
	$$X=\frac{X_{r_0s_0'}.X_{r_1s_1'}.....X_{r_ls_l'}.X_{r_{l+1}s_{l+1}'}}{X_{r_1s_1}.X_{r_2s_2}.....X_{r_{l+1}s_{l+1}}.X_{ks}},$$
	which is a stage $(l+1)$ monomial, a contradiction. So this case can't occur.\\ \\
	\noindent Case 2: $X.X_{ks}=$ some term which is not a stage-$l$-generated term, for any integer $l\geq 0$.\\
	Then the only thing possible is that $X.X_{ks}$ must get canceled with some other term of the sum $\Sigma_{j=1}^{m}g_jf_j$.\\
		If $X.X_{ks}$ gets canceled with some other term of the sum $\Sigma_{j=1}^{m}g_jf_j$, then (avoiding case 1), it can only get canceled with a term of the type $Y.X_{\alpha\beta}$, where $Y$ is again of the form of $X$ (that is, $Y$ is a term of some $g_k$ which is not a stage-$p$-monomial, for any integer $p\geq 0$). \\
	In this case, we can throw away both $X$ and $Y$ from the $g_j$'s and modify the $g_j$'s.\\
	This finishes the proof of the claim and we are done.
\end{proof}
\begin{remark}\label{r.on-mutation-thm}
It follows from Theorem \ref{t.iffhomthm} above that if polynomial mutation terminates with respect to every homogeneous component of every $f_i$, then $I=\langle f_1, \ldots, f_m\rangle$ is a homogeneous ideal. 
\end{remark}

\subsubsection{Iff conditions for polynomial mutation to terminate successfully}\label{sss.conclusion-mutation}
\begin{remark}
	Any ``term" will always mean ``Including the constant coefficient''. Same is true for generated term at any stage.\\
\end{remark}
\vspace{-10pt}
Polynomial mutation terminates successfully w.r.t $\tilde{f}_i$ and $\{f_1,\dots,f_m\}$ \textbf{iff} $\exists N\geq 0$ integer such that the following holds:\\ \\
(a)\label{mut.con.a} For some choice of terms at stages $\leq N$, any stage $N'$(where $0\leq N'\leq N$) generated term present at stage $N$ in equation $(*N)$ is either equal to a term of $\tilde{f}_i$ and every term of $\tilde{f}_i$ is obtainable in this way \textbf{OR} it cancels out with a stage $m$ generated term for some $m\in \{0,\dots,N\}$.\\ \\
and\\ \\
(b)\label{mut.con.b} If $N>1$, then $\forall$ stages $l$ ( where $l\in \{0,1,\dots,N-1\}$), we are in case $4$ only (excluding the ``polynomial mutation stopping abruptly due to cancellation'' subcase). That is, for each stage $l$( where $l\in \{0,1,\dots,N-1\}$), and for some choices of terms at stages $\leq l$ , $\exists$ at least one stage $j$ generated term (which is not a term of $\tilde{f}_i$) where $j\in\{0,\dots ,l\}$, and for each such stage $j$ generated term, $\exists$ at least one term in $f_1,\dots,f_m$(other than the $(j+1)$th multiplier of stage $j$, present inside it), which is a divisor.
\begin{remark}\label{r.canceleq}
Under the notation of Theorem \ref{t.iffhomthm}, cancellation of terms happen in the RHS of the equation $\tilde{f_i}=\Sigma_{j=1}^{m}g_jf_j$  if and only if the following type of equations hold true:
$$\mbox{A stage}\ r\ \mbox{generated term of}\ \{f_1,\ldots,f_m\}\ \mbox{originating from}\ X_{ij}=\mbox{Negative of another}$$
$$\mbox{stage}\ s\ \mbox{generated term of}\ \{f_1,\ldots,f_m\}\ \mbox{originating from}\ X_{ij'},$$
where $X_{ij}$ and $X_{ij'}$ are some terms of $\tilde{f_i}$, and $r,s$ are some integers $\geq 0$.
\end{remark}
\begin{defi}\label{d.canceleq}
We call any equation of the above type (as given in Remark \ref{r.canceleq}) as a ``cancellation equation''. 
\end{defi}
\subsection{Examples on polynomial mutation process}\label{ss.mutation-examples}
\begin{example}\label{e.mutation-stops}
Example of mutation terminating.\\
	Let, $n=7, v=w=1236457$. Here $v$ is $321$ avoiding. So, we are in Neye's set, and hence, $I_{v,w}$ is homogeneous. We will show this by using polynomial mutation now.\\
	$E_w$( the essential set corresponding to $w$)=$\{(2,1), (3,2), (4,3), (5,5), (7,6)\}$.
	\[
	Z^{(v)}= \begin{bmatrix}
		1&0&0&0&0&0&0\\
		z_{61}&1&0&0&0&0&0\\
		z_{51}&z_{52}&1&0&0&0&0\\
		z_{41}&z_{42}&z_{43}&0&1&0&0\\
		z_{31}&z_{32}&z_{33}&0&z_{35}&1&0\\
		z_{21}&z_{22}&z_{23}&1&0&0&0\\
		z_{11}&z_{12}&z_{13}&z_{14}&z_{15}&z_{16}&1\\
	\end{bmatrix}
	\]
	The only relevant minors are:\\
	\[
	(1)\rightarrow
	\quad
	\begin{bmatrix}
		z_{21}&1\\
		z_{11}&z_{14}\\
	\end{bmatrix},
	\begin{bmatrix}
		z_{23}&1\\
		z_{13}&z_{14}\\
	\end{bmatrix}
	,
	\begin{bmatrix}
		z_{22}&1\\
		z_{12}&z_{14}\\
	\end{bmatrix}
	\] corresponding to $(5,5)\in E_w$. and,
	\[
	(2)\rightarrow
	\quad
	\quad
	z_{11}, z_{14}
	\] corresponding to $(7,6)\in E_w$.
	
	Notice that in $(1)$, all the defining minors are inhomogeneous. Call their determinants $f_\alpha,f_\beta,f_\gamma$	respectively. In $(2)$, let, $z_{11}=f_a$, $z_{14}=f_b$. We will show that each homogeneous component of $f_\alpha$ belongs to $I_{v,w}$.\\
	Here, $f_\alpha=z_{21}z_{14}-z_{11}$.\\
	Let, $\tilde{f}_\alpha=z_{21}z_{14}$, and $\tilde{\tilde{f}}_\alpha=-z_{11}$. Clearly, $\tilde{f}_\alpha$ and $\tilde{\tilde{f}}_\alpha$ are two homogeneous components of $f_\alpha$.\\
	Observe that $f_a=z_{14}$ divides $z_{21}z_{14}=\tilde{f}_\alpha$.\\
	Define $\tilde{g}_a^{(0)}=\frac{z_{21}z_{14}}{z_{14}}=z_{21}$. Define all other $\tilde{g}_j^{(0)}$'s to be 0. Then,
	\[
	\sum_j \tilde{g}_j^{(0)}f_j=\tilde{g}_a^{(0)}f_a=z_{21}z_{14}=\tilde{f}_\alpha.
	\]
	So, polynomial mutation terminates successfully after stage 0 for $\tilde{f}_\alpha$. Hence, $\tilde{f}_\alpha\in I_{v,w}$.\\
	For $\tilde{\tilde{f}}_\alpha$, it is clear that $\tilde{\tilde{f}}_\alpha\in I_{v,w}$ because,
	$\tilde{\tilde{f}}_\alpha=-z_{11}$ and $z_{11}$ itself is a defining essential minor.\\
	This proves that all the homogeneous components of $f_\alpha$ belongs to $I_{v,w}$.\\
	Similarly, we can show for $f_\beta$ and $f_\gamma$. 
\end{example}
\begin{example}\label{e.mutation-special}
In this example, we will show for a homogeneous ideal (which is not necessarily a KL ideal) in $\mathbb{C}[x,y,z]$ which is generated by inhomogeneous elements, the process polynomial mutation terminates. In this example, we will \textbf{also} observe that in stage 0 of polynomial mutation, we can take a divisor of a term to be the term itself.\\
Let, $K=\mathbb{C}$ and consider $K[x,y,z]$.\\
Let $f_1=x-y^2, f_2=y^2+z, f_3=y^2$ be $3$ polynomials in $K[x,y,z]$. Let, $I=<f_1,f_2,f_3>$.\\
Clearly, $I$ is generated by $f_1,f_2,f_3$ out of which, $f_1$ and $f_2$ are inhomogeneous. But observe that $I$ is  a homogeneous ideal in $K[x,y,z]$ (because it is easy to see that in fact, $I=<x,y^2,z>$).\\
Let us look at $f_1=x-y^2$, which is an inhomogeneous polynomial. \\
Let, $\tilde{f}_1=x$.\\
Then, $\tilde{f}_1$ is a homogeneous component of $f_1$. Note that the only term from $f_1,f_2,f_3$ which divides $\tilde{f}_1$ is $x=\tilde{f}_1$ itself. So, we define:\\
$\tilde{g}_1^{(0)}=\frac{x}{x}=1$, $\tilde{g}_2^{(0)}=0=\tilde{g}_3^{(0)}$. Then,
\begin{equation}\label{**}
	\sum_{j=1}^3 \tilde{g}_j^{(0)}f_j=1.(x-y^2)+0.f_2+0.f_3=x-y^2=\tilde{f}_1-y^2
\end{equation}
Here $-y^2$ is a stage 0 generated term. Clearly, if we take the term $-y^2$ of $f_1$ to divide the $-y^2$ in \hyperref[**]{(1)}, we get :\\
$\tilde{g}_1^{(1)}=1-\frac{-y^2}{-y^2}=1-1=0$, $\tilde{g}_2^{(1)}=0=\tilde{g}_3^{(1)}$.\\
Then all $\tilde{g}_j^{(1)}(j=1,2,3)$ are $0$ and we can't conclude anything. So, instead, we take the term $y^2$ of $f_3$ as a divisor of he term $-y^2$ in equation \hyperref[**]{(1)}.\\
Then, we get,\\
$\tilde{g}_1^{(1)}=1$, $\tilde{g}_2^{(1)}=0$, $\tilde{g}_3^{(1)}=0-\frac{-y^2}{y^2}=1$ and,\\
\begin{equation}
	\sum_{j=1}^3 \tilde{g}_j^{(0)}f_j=1.(x-y^2)+0.(y^2+z)+1.(y^2)=x-y^2+y^2=x=\tilde{f}_1.
\end{equation}
Hence, the polynomial w.r.t $\tilde{f}_1$ terminates successfully after stage 1. Similarly, we can do for the homogeneous components $z$ of $f_2$, $y^2$ of $f_2$ and $-y^2$ of $f_1$. 
\end{example}
\begin{example}\label{e.mutation-notstops}
Example (Mutation not terminating, inhomogeneous)\\
Let, $n=7$, $v=2136457,w=7234615$, easy to see that $\tilde{R}_v\leq \tilde{R}_w$ and $E_w=\{(3,2),(4,3),(5,4),(5,6)\}$.\\
Let, \[
Z^{(v)}=
\begin{bmatrix}
	0&1&0&0&0&0&0\\
	1&0&0&0&0&0&0\\
	z_{51}&z_{52}&1&0&0&0&0\\
	z_{41}&z_{42}&z_{43}&0&1&0&0\\
	z_{31}&z_{32}&z_{33}&0&z_{35}&1&0\\
	z_{21}&z_{22}&z_{23}&1&0&0&0\\
	z_{11}&z_{12}&z_{13}&z_{14}&z_{15}&z_{16}&1\\
\end{bmatrix}
\]
In this example, one of the defining minors is 
\[
\begin{bmatrix}
	0&z_{35}&1\\
	1&0&0\\
	z_{14}&z_{15}&z_{16}\\
\end{bmatrix}
\]
which is having determinant $f_\alpha=z_{15}-z_{35}z_{16}$. Clearly, $f_\alpha$ is an inhomogeneous minor.\\
Observe that it is not possible to divide the term $z_{15}$ of $f_\alpha$ by any monomial of any other defining essential minor. Hence, by contrapositive of the homogeneity lemma (Lemma \ref{l.homthm}), it follows that the $KL$ ideal $I_{v,w}$ is inhomogeneous.\\
Let us now see that in this example the polynomial mutation process does not terminate successfully w.r.t some homogeneous component of some inhomogeneous defining essential minor.\\
Observe that $\begin{bmatrix}
	z_{21}&1\\
	z_{11}&z_{14}\\
\end{bmatrix}
$ is a defining essential minor of $I_{v,w}$ here, whose determinant $f_1$ is itself an inhomogeneous polynomial.\\
$f_1=z_{21}z_{14}-z_{11}$.\\
Let, $\tilde{\tilde{f}}_1=-z_{11}$. Then $\tilde{\tilde{f}}_1$ is homogeneous component of $f_1$. We will now show that polynomial mutation process stops abruptly w.r.t $\tilde{\tilde{f}}_1$ and all the defining essential minors, if we go via one route.\\
Route 1:\\
Notice that the term $-z_{11}$ of $f_1$ divides itself. So, define $\tilde{g}_1^{(0)}=\frac{-z_{11}}{-z_{11}}=1$. Define all other $\tilde{g}_j^{(0)}$'s to be 0. Then,\\
$\sum_j \tilde{g}_j^{(0)}f_j=\tilde{g}_1
^{(0)}f_1=1.(z_{21}z_{14}-z_{11})=\tilde{\tilde{f}}_1+z_{21}z_{14}$.\\
Then, let us use $f_2=det
\begin{bmatrix}
	z_{31}&0&1\\
	z_{21}&1&0\\
	z_{11}&z_{14}&z_{16}\\
\end{bmatrix}=-z_{11}+z_{14}z_{21}+z_{31}z_{16}$ for mutation in the next stage.\\
Clearly, the term $z_{14}z_{21}$ of $f_2$ divides the term $z_{21}z_{14}$ of $\sum_j \tilde{g}_j^{(0)}f_j$.\\
Define $\tilde{g}_2^{(1)}= \tilde{g}_2^{(0)}-\frac{z_{21}z_{14}}{z_{21}z_{14}}=0-1=-1,\tilde{g}_1^{(1)}=\tilde{g}_1^{(0)}=1$ and all other $\tilde{g}_j^{(1)}$'s as zero.\\
Then, $\sum_j \tilde{g}_j^{(1)}f_j=1.f_1-1.f_2=z_{21}z_{14}-z_{11}-(-z_{11}+z_{14}z_{21}+z_{31}z_{16})=-z_{31}z_{16}$. \\
Now consider the defining essential minor $\begin{bmatrix}
	z_{31}&0\\
	z_{21}&1\\
\end{bmatrix}$.\\
Let $f_3$ denote its determinant. Then, $f_3=z_{31}$.\\
Clearly, the term $z_{31}$ of $f_3$ divides the term $-z_{31}z_{16}$ of $\sum_j\tilde{g}_j^{(1)}f_j$.\\
Define $\tilde{g}_3^{(2)}=\tilde{g}_3^{(1)}-\frac{-z_{31}z_{16}}{z_{31}}=0+z_{16}=z_{16}$, $ \tilde{g}_2^{(2)}=\tilde{g}_2^{(1)}=-1$ and $\tilde{g}_1^{(2)}=\tilde{g}_1^{(1)}=1$ and all other $\tilde{g}_j^{(2)}$'s to be 0. Then,\\
$\sum_j \tilde{g}_j^{(2)}f_j=1.f_1-1.f_2+z_{16}.f_3=0$. Hence, mutation stops abruptly.\\
We can go similarly in different routes(there are only finitely many routes available for $\tilde{\tilde{f}}_1$) and show that polynomial mutation does not terminate successfully w.r.t $\tilde{\tilde{f}}_1$ in any of the chosen routes. This was expected to happen since the $KL$ ideal $I_{v,w}$ here is inhomogeneous. 
\end{example}
\section{Conditions for homogeneity of KL-ideals}\label{s.iff}
In this section, we will provide some sufficient conditions for the homogeneity of a KL-ideal $I_{v,w}$. Furthermore, we also show that under a certain assumption on $v$ and $w$, these sufficient conditions are also necessary.\par
Let us first recall the following definition.
\begin{defi}\label{d.tilde-order}
We say that $v\leq w$ in the tilde matrix order if $\tilde{R}_v\geq\tilde{R}_w$. 
\end{defi}

\subsection{Moving from $v,w$ to $v',w'$}\label{ss.moving-to-prime}
\begin{remark}\label{r.vwtov'w'}
	According to the convention used in Neye's paper \cite{neye2023grobner}, we have :$w\leq v$ in $S_n \iff R_w \geq R_v$ in terms of the usual rank matrix defined in \S\ref{s.Introduction} and the usual order $\leq$ on matrices (see Definition \ref{d.AnleqB}).\par 
	Let $w'(i)=n-w(i)+1$ and $v'(i)=n-v(i)+1 $, $\forall\ i\in \{1,2,\dots, n\}$.\\Again, for all $i,j\in\{1,\ldots,n\}$, observe that: $r_{ij}^{(w)} = \tilde{r}_{n-i+1,j}^{(w')}$. Then,
	\begin{align*}
	 w\leq v\ (\mbox{In Neye's convention})\ &\iff R_w \geq R_v\\
	 &\iff r_{ij}^{(w)} \geq r_{ij}^{(v)} \forall i,j \\
	&\iff \tilde{r}_{(n-i+1,j)} \geq \tilde{r}_{n-i+1,j}^{(v')} \forall i,j\\ 
	&\iff \tilde{r}_{a,b}^{(w')} \geq\tilde{r}_{a,b}^{(v')} \forall a,b \in \{1,\dots , n\}.\\
	&\iff \tilde{R}_{w'} \geq \tilde{R}_{v'}\\
	&\iff w' \leq v'  \textnormal{(in the tilde matrix order)}
\end{align*}
	
	So, we conclude that $w\leq v$ (according to the convention in Neye's paper) if and only if $w' \leq v'$(in the tilde matrix order). 
\end{remark}
\begin{lemma}\label{l.Ivw=Iv'w'}
$I_{v,w}$(in Neye's convention)= $I_{v',w'}$(in our tilde matrix convention).\end{lemma}
\begin{proof}
	Consider a $s\times t$ northwest sub-matrix of ambient matrix $Z^{(v)}$ in Neye's convention. where we need to take some $\alpha$ minors that will become part of generators of $I_{v,w}$. Then, $\alpha=r_{s,t}^{(w)}+1$ under the convention of Neye. Now, consider $\tilde{R}_{w'}$ and notice that $\tilde{r}_{n-s+1,t}^{(w')}=r_{s,t}^{(w)}=\alpha-1$ where $w'=n-w+1$, i.e, $w'(i)=n-w(i)+1$ $\forall$ $1\leq i \leq n$. So, we are supposed to take $\alpha-1+1=\alpha$ minors in the southwest $(n-s'+1)\times t'$ submatrix of $Z^{(v')}$ where $s',t'$ are the row  and column subscripts of the corresponding $\tilde{r}_{s',t'}^{(w')}$ value, which here is $s'=n-s+1,t'=t$. This implies, $n-s'+1=n-(n-s+1)+1=s$, so we are to take $\alpha$ minors in the southwest $s\times t$ submatrix of $Z^{(v')}$.
	\par
	Now, recall that our $Z^{(v')}$ is the mirror image (vertical flip) of the Neye's $Z^{(v)}$ matrix,i.e, the first row of $Z^{(v)}$ in convention of Neye is the last row of the matrix $Z^{(v')}$
	in our convention and same for the successive rows(the second row is our second last row and so on). Also, the pattern of theirs : every matrix entry to the bottom and right of 1 is 0 is vertical flip of ours : every matrix entry to the top and right of 1 is 0.\\
	So,taking $\alpha$ minor in their $Z^{(v)}$ matrix's northwest $s\times t$ submatrix is same as taking $\alpha$ minor in the southwest $s\times t$ submatrix of our $Z^{(v')}$ and so, if $\{f_1,\dots, f_k\}$ is the set of generators obtained by taking $\alpha$ minors above, then we too get $\{f_1,\dots, f_k\}$ as $\alpha$ minors of our corresponding submatrix.\\
	Also, as $r_{i,j}^{(w)}=\tilde{r}_{n-i+1,j}^{(w')}$ $\forall$ $i,j$ (same for $v'$), so the same argument can be mimicked starting with any such $\alpha$ minors of $Z^{(v')}$ matrix and moving to $Z^{(v)}$. Hence, the set of generators of $I_{v,w}$ is same as the set of generators of $I_{v',w'}$, which thus implies,\\
	$I_{v,w}$(in Neye's convention)=$I_{v',w'}$(in our tilde matrix convention).
\end{proof}
\begin{remark}\label{r.Ivw=Iv'w'}
It follows from Lemma \ref{l.Ivw=Iv'w'} that $I_{v,w}$ is homogeneous if and only if $I_{v',w'}$ is so. 
\end{remark}
Let us now rephrase \cite[Proposition 6.4]{neye2023grobner} in terms of $v'$ and $w'$, which is given in Proposition \ref{p.Neye} below:
\begin{pro}\label{p.Neye}
Let $(v,w) \in S_n \times S_n$ for which $w' \leq v'$ in the tilde matrix order.\\
	(a) If $v'$ is $123$ avoiding or $w'$ is $312$ avoiding then $I_{v',w'}$(and hence $I_{v,w}$) is standard homogeneous).\\
	(b) Suppose $v'$ is non $123$ avoiding and $w'$ is $312$ avoiding then $I_{v',w'}$(and hence $I_{v,w}$) is standard homogeneous if for all pairs $(a'_1a'_2a'_3 , c'_1c'_2c'_3)$ either $c'_3 > a'_2$ or $w'^{-1}(c'_2) < v'^{-1}(a'_2)$, where $a'_1a'_2a'_3$ is a $123$ pattern in $v'$ and $c'_1c'_2c'_3$ is a $312$ pattern in $w'$. 
\end{pro}	
Notice that by Theorem \ref{t.wholeRing}, we do not need to consider the $w' \nleq v'$ case in the tilde matrix order.
\subsection{A $k$ band}\label{ss.k-band}
\begin{defi}\label{d.k-band}
Given any square matrix $A$ of size $p$ with set of rows $R= \{r_1,\dots,r_p\}$ and columns $ S=\{s_1,\dots,s_p\}$, we define a $k$ band of rows of $A$ (where $k\in\{1,\ldots,p\}$) as the collection of all rows corresponding to a $k$ size subset of $R$. Similarly one can define a band of $k$ columns of $A$ where $k\in \{1,\dots,p\}$.\end{defi}
\begin{example}\label{e.kbandeg}
For a $7\times 7$ matrix $A=[(z_{ij})]$ , the collection of $3$ rows given by $(z_{31}, \dots ,z_{37}),(z_{51}, \dots ,z_{57})$ and $(z_{71}, \dots ,z_{77})$ is an example for a $k=3$ band.
\end{example}
\begin{defi}\label{d.kminor}
	A $k$ minor inside a $k$ band of rows is simply a $k$ minor inside the $k \times p$ matrix defined by the k band. In case of $k$ band of columns, it is a $k$ minor inside the $p\times k$ matrix defined by the $k$ band.
\end{defi}	
\begin{example}\label{e.kminor}
Let $A$ be as in Example \ref{e.kbandeg} and consider the 3 band there. Any 3 minor of the $3\times 7$ submatrix defined by the $3$ bands of rows, is a 3 minor inside the 3 band. 
\end{example}
\begin{pro}\label{p.kband}
Let $A$ be any inhomogeneous minor of $Z^{(v')}$. Let $M$ be the subminor of $A$ obtained by deleting all the rows and columns from $A$ which contain $1$. Let $p,m$ denote the sizes of the minor $A$ and its subminor $M$ respectively. Assume that $M$ is non singular. Let $k\in \{1,\dots, m\}$ be arbitrary but fixed. Then, there exists a band of $k$ rows or a band of $k$ columns in $A$ such that all the $k\times k$ minors inside the band generate any homogeneous component of $det(A)$.
\end{pro}
\begin{proof} We prove this for $k=m$ and the general case will follow from that.\\
Let $A$ be a $p$ minor with rows $D=\{d_1,\dots,d_p\} $ and columns $B=\{b_1,\dots, b_p\}$ and,
$R=\{r_1,\dots, r_t\}$, $S=\{s_1,\dots , s_t\}$ be the rows and columns that have 1, that is, $z_{r_is_i}=1  $ $\forall$ $1\leq i \leq t$. Then, the remaining set of rows $D\setminus R$ and columns $B\setminus S$ form the $m\times m$ matrix $M$ of only zeroes and variables $(m=p-t)$.\\ \\
Now, we can write $det(A)$ as the sum of its homogeneous components where the lowest degree (degree $m$) homogeneous component,call it $f_m$, is $det(M)$. Clearly, $f_m$ belongs to the $m$ band of $A$ made up of the columns corresponding to the set $B\setminus S$.\\ \\
Now, consider any $l^{th}$ homogeneous component (call them $f_l$) of $ det(A)(m<l\leq p)$(if any). It will be obtained by taking exactly $(l-m)$ indeterminates from the columns in $S$, and taking 1 from the rest of the columns in $S$, and concatenating with the $m$-matrix formed by the remaining rows not chosen and columns in $B\setminus S$, call it $M_{S'}$ where $S'$ is a $l-m$ cardinality subset of $S$.\par 
Notice that $M_{S'}$ is a minor inside the $m$-band made of the columns in $B\setminus S$. Thus, we have $f_l= \sum_{S'}z_{\alpha_1\beta_1}\dots z_{\alpha_{l-m}\beta_{l-m}}\cdot det(M_{S'})$ where the sum runs over all possible $l-m$ cardinality subsets $S'$ of $S$ and the placeholders in the summation are $S'=\{\beta_1,\dots \beta_{l-m}$\} and the rows $\alpha_i$ corresponding to $S'$. Hence, $f_l$ lies in the $\mathbb{C}[z_{ij}]$ linear combination of $det(M_{S'})$, which are $m$-minors in the $m$-band made up of the columns $B\setminus S$.\par
So, if $\mathcal{M}$ denotes the set of all $m$ minors that can be formed in the $m$-band made of the columns $B\setminus S$, then the above paragraph shows that every degree $l$ homogeneous component lies in the $\mathbb{C}[z_{ij}]$ span of $\mathcal{M}$ and since $det(A) $ is the signed sum of all its homogeneous components, hence $det(A)$ also lies in the $\mathbb{C}[z_{ij}]$ span of $\mathcal{M}$.
\\~\\
Notice that any $k$-sub-band $(k\leq m)$ of $S\setminus B$ also works by same logic.
\end{proof}
\begin{remark}\label{r.kband-unique}
  (On the uniqueness or non-uniqueness of a $k$-band): Any $l^{th}$ homogeneous component (call it $f_l$) of $ det(A)(m<l\leq p)$(if any) is of the form:  $f_l= \sum_{S'}z_{\alpha_1\beta_1}\dots z_{\alpha_{l-m}\beta_{l-m}}\cdot det(M_{S'})$ where $M_{S'}$ is defined as above and the sum runs over all possible $l-m$ cardinality subsets $S'$ of $S$ and the placeholders in the summation are $S'=\{\beta_1,\dots \beta_{l-m}$\} and the rows $\alpha_i$ corresponding to $S'$. Hence, in particular, to get the $m^{th}$ homogeneous component (if it exists), any $m$ band needs to have only the contribution of variables appearing in $f_m$ and nothing else. Now, this matrix $M$ is non singular and hence, has a path of degree $m$, which in turn forces us to choose the unique $m$ band made up of the rows or columns (depending on whether we are talking of $m$ band of rows or columns) of $M$.\par
As for $k\in \{1,\dots, m\}$ we always have $\binom{m}{k}$ choices, so is not unique.
\end{remark}
\begin{remark}\label{r.for-assumption}
 In the light of Proposition \ref{p.kband} and Remark \ref{r.kband-unique}, for any $k\in\{1,\ldots,m\}$, there always exist finitely many $k$ bands in $A$ such that all the $k\times k$ minors inside it generate any homogeneous component of $det(A)$. Let us denote any such $k$ band by the symbol $\overline{Band}^{k}_{i}$, where $i$ runs over a finite set.
\end{remark}

\subsection{Some notation and terminology}\label{ss.notation-new}
	Consider the minor $A$ (as in Proposition \ref{p.kband}) which is of size $p$. Let $M$ be as mentioned in the same Proposition. 
	\begin{defi}\label{d.foreign}
	Any variable of the $Z^{(v')}$ matrix other than the variables in $M$ is called a foreign variable. 
	\end{defi}
	\noindent\textbf{The collection $\widehat{M}$:} Consider the collection of all defining essential minors of the $Z^{(v')}$ matrix (see Remark \ref{r.essentialset} for the definition of a defining essential minor) which \textbf{intersect} $M$ \textbf{non trivially} (i,e, if some of he row indices as well as column indices of $M$ coincide with the row and column indices of the defining essential minor of $Z^{(v')}$ under consideration). Then consider all terms from this collection of minors (determinants) which contain at least one foreign variable. Call this collection of terms as $\widehat{M}$. [Here, a term is always considered along with its sign factor.]
	\begin{defi}\label{d.multiset}
	Let $X$ be any set. A multiset $Y$ in $X$ is defined to be a function $Y:X\rightarrow\{0,1,2,\ldots\}$. One should think of $Y$ as consisting of the set $X$ of elements, but with each $x\in X$ occurring $Y(x)$ times.
	\end{defi}
	A multiset in $\widehat{M}$ is defined in the same way. 
	\begin{example}\label{e.multiset}
	Let $X$ be the set $\{z_{11},z_{12}\}$. A multiset in $X$ is $\{z_{11},z_{11},z_{12}\}$ whereas, a set will forbid the repetition  of elements and will just be $\{z_{11}, z_{12}\}$. Note that a set is also a multiset of a special type, where each element occurs exactly once.
	\end{example}
	\begin{defi}\label{d.equal-multisets}
	We say two multisets (in some set $X$) are equal if they are equal element wise and counting repetitions.
	\end{defi}
	We adapt the same definition for $\widehat{M}$.
	\begin{remark}\label{r.multiset}
	Notice that we may have two multisets that are unequal but their underlying sets are equal.
	\end{remark}
	\begin{example}\label{e.multiset-remark}
	Let $X$ be as in Example \ref{e.multiset}. Then, $\{z_{11},z_{11}\}$ and $\{z_{11}\}$ are not equal as multisets as $z_{11}$ occurs twice in one multiset and once in another. However, their underlying sets are equal.
	\end{example}
	\noindent\textbf{The collection $\widetilde{M}$:} Consider the collection of all defining essential minors of $Z^{(v')}$ which intersect $M$ non trivially. Then consider all terms from this collection of minors/determinants, which are free of foreign variables. Call this collection as $\widetilde{M}$. [Here also, a term is considered along with its sign factor.]\par 
	A multiset in $\widetilde{M}$ is defined similarly.
	\begin{pro}\label{p.statement3}
	Let $A$ and $M$ be as in Proposition \ref{p.kband}. If conditions (a1) and (a2) below both hold, where:\\
	(a1) Consider any two distinct multisets in $\widehat{M}$ . Let $a$ and $b$ denote the product of these two multisets. Then neither $\frac{a}{b}=-1$ nor $\frac{a}{b}$ is free of foreign variables.\\
	(a2) Consider any two multisets in $\widetilde{M}$ (need not be distinct). Let $c$ and $d$ denote the product of these two multisets. Then, $\frac{c}{d}\neq -1$.\\
	Then $det(M)\notin I_{v',w'}$ and hence $I_{v',w'}$ is inhomogeneous.
	\end{pro}
	\begin{proof}
	The proof is by contradiction method. Suppose $det(M)$ belongs to $I_{v',w'}$. \\
	It is easy to observe that $det(M)$ is the lowest degree homogeneous component of $det(A)$. Since $det(M)$ belongs to $I_{v',w'}$, we must have that:\\
	For some terms $X_{ij}$ and $X_{ij'}$ of $det(M)$, the following equation holds true (by Remark \ref{r.canceleq}):\\
	$$\frac{X_{ij}\Pi_1}{\Pi_2}=\frac{-X_{ij'}\Pi_3}{\Pi_4},$$
	where the $\Pi_i$'s ($i=1,2,3,4$) are certain products of terms of certain defining essential minors of $Z^{(v')}$ (which intersect $M$ non trivially). Observe that the phrase ``which intersect $M$ non trivially'' is present here because the defining essential minors of $Z^{(v')}$ which do not intersect $M$ non trivially, are in fact, redundant as generators for $det(M)$.\par 
	The above equation can be rewritten as:\\
	$$X_{ij}\Pi_1\Pi_4=-X_{ij'}\Pi_3\Pi_2.$$
	In this equation, look at the $\Pi_i$'s. Inside the $\Pi_i$'s, if there are any terms which purely consist of variables from $M$, then we pull them out from both the LHS and the RHS. Say, the resulting equation becomes:\\
	$$X_{ij}m_1\Pi_1'\Pi_4'=-X_{ij'}m_2\Pi_3'\Pi_2',$$
	where $m_1$ and $m_2$ are products of the pulled out terms on both sides, and the $\Pi_i'$'s are the products of terms of the $\Pi_i$'s which contain at least one foreign variable in them.\\
	Now if $\Pi_1'\Pi_4'\neq \Pi_3'\Pi_2'$, we are done with the proof by condition (a1).\\
	In the remaining case, we have:\\
	$$1=\frac{-X_{ij'}m_2}{X_{ij}m_1}.$$
	We are now appeal to condition (a2) for arriving at a contradiction.
\end{proof}	
The following example illustrates Proposition \ref{p.statement3} above.
\begin{example}\label{e.statement3}
	Let $n=7$, $v'=2136457$ and $w'=7234615$. Then the corresponding $\tilde{R}_{v'}$ and $\tilde{R}_{w'}$ matrices are:\\
	\[
	\tilde{R}_{v'}= \begin{bmatrix}
		1&2&3&4&5&6&7\\
		1&1&2&3&4&5&6\\
		0&0&1&2&3&4&5\\
		0&0&0&1&2&3&4\\
	    0&0&0&1&1&2&3\\
		0&0&0&1&1&1&2\\
		0&0&0&0&0&0&1\\
	\end{bmatrix}	
	\quad
	\textnormal{and} 
	\quad 	
	\tilde{R}_{w'}=\begin{bmatrix}
		1&2&3&4&5&6&7\\	
		1&2&3&4&5&5&6\\
		1&1&2&3&4&4&5\\
		1&1&1&2&3&3&4\\
		1&1&1&1&2&2&3\\
		1&1&1&1&2&2&2\\
		1&1&1&1&1&1&1\\	
	\end{bmatrix}		
	\]
Notice that $w'\leq v'$. The corresponding $Z^{(v')}$ matrix is :\\
\[
Z^{(v')}=\begin{bmatrix}
	0&1&0&0&0&0&0\\
	1&0&0&0&0&0&0\\
	z_{51}&z_{52}&1&0&0&0&0\\
	z_{41}&z_{42}&z_{43}&0&1&0&0\\
	z_{31}&z_{32}&z_{33}&0&z_{35}&1&0\\
	z_{21}&z_{22}&z_{23}&1&0&0&0\\
	z_{11}&z_{12}&z_{13}&z_{14}&z_{15}&z_{16}&z_{17}\\
\end{bmatrix}
\]
Then it is easy to verify that $(5,6)\in E_{w'}$, and corresponding to the point $(5,6)$ of $E_{w'}$, we have an inhomogeneous defining essential minor given by:\\
\[
A=\begin{bmatrix}
	0&z_{35}&1\\
	1&0&0\\
	z_{14}&z_{15}&z_{16}\\
	\end{bmatrix}
		\quad
	\textnormal{for which} 
	\quad 
	M=\begin{bmatrix}
		z_{15}
	\end{bmatrix}.	
	\]
	And all the defining essential minors that intersect $M$ non trivially are :\\
	\[
	\begin{bmatrix}
		z_{33}&z_{35}&1\\
		z_{23}&0&0\\
		z_{13}&z_{15}&z_{16}\\
	\end{bmatrix},
	\begin{bmatrix}
	z_{32}&z_{35}&1\\
	z_{22}&0&0\\
	z_{12}&z_{15}&z_{16}\\
	\end{bmatrix},
	\begin{bmatrix}
	z_{31}&z_{35}&1\\
	z_{21}&0&0\\
	z_{11}&z_{15}&z_{16}\\
	\end{bmatrix},
	\],
	\[
	\begin{bmatrix}
		z_{31}&z_{32}&z_{35}\\
		z_{21}&z_{22}&0\\
		z_{11}&z_{12}&z_{15}\\
	\end{bmatrix},
\begin{bmatrix}
	z_{31}&z_{33}&z_{35}\\
	z_{21}&z_{23}&0\\
	z_{11}&z_{13}&z_{15}\\
\end{bmatrix},	
	\begin{bmatrix}
		z_{32}&z_{33}&z_{35}\\
		z_{22}&z_{23}&0\\
		z_{12}&z_{13}&z_{15}\\
	\end{bmatrix},
	\],
	\[
\begin{bmatrix}
	z_{31}&0&z_{35}\\
	z_{21}&1&0\\
	z_{11}&z_{14}&z_{15}\\
\end{bmatrix},	
	\begin{bmatrix}
		z_{32}&0&z_{35}\\
		z_{22}&1&0\\
		z_{12}&z_{14}&z_{15}\\
	\end{bmatrix},	
	\begin{bmatrix}
		z_{33}&0&z_{35}\\
		z_{23}&1&0\\
		z_{13}&z_{14}&z_{15}\\
	\end{bmatrix},
	A=\begin{bmatrix}
	0&z_{35}&1\\
	1&0&0\\
	z_{14}&z_{15}&z_{16}\\
	\end{bmatrix}
	\]
with determinants:\\
 $z_{23}z_{16}z_{35}-z_{23}z_{15},\\
 z_{22}z_{16}z_{35}-z_{22}z_{15},\\
z_{21}z_{16}z_{35}-z_{21}z_{15},\\
z_{35}z_{21}z_{12}-z_{35}z_{11}z_{22} + z_{15}z_{31}z_{22}-z_{15}z_{21}z_{32},\\
z_{35}z_{21}z_{13}-z_{35}z_{11}z_{23} + z_{15}z_{31}z_{23}-z_{15}z_{21}z_{33},\\
z_{35}z_{22}z_{13}-z_{35}z_{12}z_{23} + z_{15}z_{32}z_{23}-z_{15}z_{22}z_{33},\\
z_{21}z_{14}z_{35}-z_{11}z_{35}+z_{31}z_{15},\\
z_{22}z_{14}z_{35}-z_{12}z_{35}+z_{32}z_{15},\\
z_{23}z_{14}z_{35}-z_{13}z_{35}+z_{33}z_{15},\\
z_{15}-z_{35}z_{16}$\\
respectively.\\
It is easy to see by inspection that conditions (a1) and (a2) of Proposition \ref{p.statement3} above hold. Hence it follows from the same Proposition that $det(M)=z_{15}\notin I_{v',w'}$, which implies that $I_{v',w'}$ is inhomogeneous. \\
We can also verify easily that $det(M)=z_{15}$ cannot belong to $I_{v',w'}$, because for it to happen, we must have that the term $-z_{35}z_{16}$ cancels out with something. But the only way it cancels out with something is: there should have existed solo terms $\pm z_{35}$ or $\pm z_{16}$ in defining essential minors of $Z^{(v')}$. And this doesn't happen here. 
\end{example}
\subsection{The conditions for homogeneity}\label{ss.the-iff-conditions}
In this subsection, we give some sufficient conditions for the homogeneity of a KL-ideal $I_{v',w'}$, and also show that under an assumption on $v',w'$ (that is, for a restricted class of $v,w$ in $S_n$), these conditions are also necessary.\par
Recall the notation $\overline{Band}^k_i$ from Remark \ref{r.for-assumption}. We will now state the assumption on $v',w'$.\\ 
\noindent\textbf{The assumption (***) on} $v',w'$: Let $A,M,p,m$ be as in the Proposition \ref{p.kband}. Let $v',w'$ be such that given any $k\in\{1,\ldots,m\}$,
\begin{quote}\label{assumption}
Any minimal collection (minimal in the sense of containment) of minors of $Z^{(v')}$ that generate all the homogeneous components of $det(A)$ come only from the collection of all $k\times k$ minors of the same $k$-band ($\overline{Band}^k_i$) of $A$.\hfill(***)
\end{quote}
\begin{remark}\label{r.on-assumption}
Inside the assumption (***) on $v',w'$ that we have made above, by the phrase ``any minimal collection of minors of $Z^{(v')}$ that generate all the homogeneous components of $det(A)$'', we do not mean the defining essential minors only.
\end{remark}
\noindent\textbf{Notation:} In Proposition \ref{p.statement2andconverse} below, the notation (b)$^c$ means the complement of condition (b). Similarly for (a1 and a2)$^c$. In particular, (a1 and a2)$^c$ will mean (a1)$^c$ or (a2)$^c$.
\begin{pro}\label{p.statement2andconverse}
Let $A$, $M$, $m$ and $p$ be as in Proposition \ref{p.kband}. Assume that $A$ is inhomogeneous. Consider conditions (a1),(a2) as in Proposition \ref{p.statement3} above and let condition (b) be as mentioned below:\\ \\
\noindent (b) Given any $k\in \{1,\ldots,m\}$, and any $k$-band of $A$, either not all of the $k\times k$ minors inside this $k$-band are essential OR they (all the $k\times k$ minors inside this $k$-band) do not generate every homogeneous component of $det(A)$.\\ \\
If (b)$^c$ and (a1 and a2)$^c$ holds, then every homogeneous component of $A$ belongs to $I_{v',w'}$.\\
Furthermore, under the above mentioned assumption (***) on $v',w'$, the converse of the above result also holds true.
\end{pro}
\begin{proof}
Given the condition (b), the condition (b)$^c$ will be: 
	\begin{quote}
	There exists a $k\in\{1,\ldots,m\}$, and a $k$-band of $A$ such that all the $k\times k$ minors inside this $k$-band are essential and they generate every homogeneous component of $det(A)$.
	\end{quote}
	If (b)$^c$ and (a1 and a2)$^c$ holds, then in particular, (b)$^c$ holds. Then the proof is obvious.\par
Let us now prove the converse part under the assumption on $v',w'$. We will prove the contrapositive of the converse statement. That is, we will prove that: (b) OR (a1 and a2) imply that not every homogeneous component of $det(A)$ belongs to $I_{v',w'}$.\par 
Note that if (a1) and (a2) hold, then we are done by Proposition \ref{p.statement3}. \par 
Let us now deal with the case when (b) holds and (a1 and a2) doesn't hold. Then in particular, (b) holds. In this situation, we proceed as follows:\\
By Proposition \ref{p.kband}, we know that given any $k\in\{1,\ldots,m\}$, there exists a $k$-band of $A$ (call it $\overline{Band}^k_i$) such that all the $k\times k$ minors inside this $k$-band generate every homogeneous component of $det(A)$. Fix any such $\overline{Band}^k_i$ (Here, $i$ varies over a finite set.).\par 
Then condition (b) will imply that not all of the $k\times k$ minors inside $\overline{Band}^k_i$ are essential, and this is true for every such $\overline{Band}^k_i$ (where $i$ varies over a finite set). \par 
Consider some homogeneous component $H$ of $det(A)$, which has at least one non-essential $k\times k$ minor of a $\overline{Band}^k_i$ as a generator. Since not all of the $k\times k$ minors of $\overline{Band}^k_i$ are essential, therefore it follows from the assumption (***) on $v',w'$ that $H$ cannot belong to $I_{v',w'}$.
\end{proof}
\begin{remark}\label{r.converse-remark}
Notice that in Proposition \ref{p.statement2andconverse} above, the assumption (***) is needed only to prove the converse part. 
\end{remark}
The following example illustrates the first assertion of Proposition \ref{p.statement2andconverse} above.
\begin{example}\label{e.statement2andconverse}
For $n=6$. Let $v'=125364$ and $w'=126453$. Then we have:\\
\[
\tilde{R}_{w'}=
\begin{bmatrix}
	1&2&3&4&5&6\\
	0&1&2&3&4&5\\
	0&0&1&2&3&4\\
	0&0&1&2&3&3\\
	0&0&1&1&2&2\\
	0&0&1&1&1&1\\
\end{bmatrix}
\quad
\textnormal{and}
\quad
\tilde{R}_{v'}=
\begin{bmatrix}
	1&2&3&4&5&6\\
	0&1&2&3&4&5\\
	0&0&1&2&3&4\\
	0&0&1&1&2&3\\
	0&0&1&1&2&2\\
	0&0&0&0&1&1\\
\end{bmatrix}
\]
Here, $\tilde{R}_{w'}\geq \tilde{R}_{v'}$. Hence, $w'\leq v'$. And we have, $E_{w'}=\{(2,1), (3,2),(5,4)\}$. Corresponding to the points of $E_{w'}$, we have $\tilde{r}^{(w')}_{21}=0,\tilde{r}^{(w')}_{32}=0, \tilde{r}^{(w')}_{54}=1$.
\[
Z^{(v')}=\begin{bmatrix}
1&0&0&0&0&0\\
z_{51}&1&0&0&0&0\\
z_{41}&z_{42}&0&1&0&0\\
z_{31}&z_{32}&0&z_{34}&0&1\\
z_{21}&z_{22}&1&0&0&0\\
z_{11}&z_{12}&z_{13}&z_{14}&1&0\\	
\end{bmatrix}
\]
Now \\
$(s,t)=(2,1)\implies(n-s+1,t)=(6-2+1,1)=(5,1)$ SW submatrix of $Z^{(v')}$.\\
$\tilde{r}^{(w')}_{21}=0\implies\tilde{r}^{(w')}_{21}+1=1$.\\
 All $1\times1$ minors of the SW $(5,1)$ submatrix of $Z^{(v')}$ are:\\
$z_{11}, z_{21}, z_{31}, z_{41}, z_{51}$.
\\~\\
$(s,t)=(3,2)\implies(n-s+1,t)=(6-3+1,1)=(4,2)$ SW submatrix of $Z^{(v')}$.
$\tilde{r}^{(w')}_{32}=0\implies\tilde{r}^{(w')}_{32}+1=1$\\
 All $1\times1$ minors of the SW $(4,2)$ submatrix of $Z^{(v')}$ are:\\
$z_{11}, z_{12}, z_{21}, z_{22}, z_{31}, z_{32}, z_{41}, z_{42}$.
\\~\\
$(s,t)=(5,4)\implies(n-s+1,t)=(6-5+1,4)=(2,4)$ SW submatrix of $Z^{(v')}$.
$\tilde{r}^{(w')}_{54}=1\implies\tilde{r}^{(w')}_{54}+1=2$\\
All $2\times2$ minors of the SW $(2,4)$ submatrix of $Z^{(v')}$ are:
\[
\begin{bmatrix}
	z_{21}&z_{22}\\
	z_{11}&z_{12}\\
\end{bmatrix},
\quad
\begin{bmatrix}
	z_{21}&1\\
	z_{11}&z_{13}\\
\end{bmatrix},
\quad
\begin{bmatrix}
	z_{21}&0\\
	z_{11}&z_{14}\\
\end{bmatrix},
\]
\[
\begin{bmatrix}
	z_{22}&1\\
	z_{12}&z_{13}\\
\end{bmatrix},
\quad
\begin{bmatrix}
	z_{22}&0\\
	z_{12}&z_{14}\\
\end{bmatrix},
\quad
\begin{bmatrix}
	1&0\\
	z_{13}&z_{14}\\
\end{bmatrix}.
\]
Notice that \[
\begin{bmatrix}
	z_{21}&1\\
	z_{11}&z_{13}\\
\end{bmatrix}
\quad
\textnormal{and}
\begin{bmatrix}
	z_{22}&1\\
	z_{12}&z_{13}\\
\end{bmatrix}
\]
are the only inhomogeneous defining essential minors here with determinants $z_{21}z_{13}-z_{11}$ and $z_{22}z_{13}-z_{12}$ respectively.\\
It is easy to see by inspecting the example that $I_{v',w'}$ is homogeneous. Let us now verify this using the previous Proposition (\ref{p.statement2andconverse}).\\
Let us first look at\\
\[A=\begin{bmatrix}
	z_{22}&1\\
	z_{12}&z_{13}\\
\end{bmatrix}
\quad
\textnormal{for which}
\quad
M=\begin{bmatrix}
	z_{12}\\
\end{bmatrix}.
\]
Here, $p=2$ and $m=1$.  So $k=1$ and $(b)^c$ holds clearly.\\
Does ($a_1$ and $a_2$)$^c=a_1^c$ or $a_2^c$ hold?\\
Let us check whether or not $a_2^c$ holds:\\
Here, $M=\begin{bmatrix}
	z_{12}
\end{bmatrix}$.\\
The collection of all defining essential minors of $Z^{(v')}$ which intersect $M=\begin{bmatrix}
z_{12}
\end{bmatrix}$ non trivially are:\\
$z_{12}, z_{21}z_{12}-z_{11}z_{22}, z_{22}z_{13}-z_{12}.$
The collection of all terms in this collection of determinants, which are free of foreign variables is given by:\\
$\widetilde{M}=\{z_{12},-z_{12}\}$.\\
Then, $\exists$ two multisets (need not be distinct) in $\widetilde{M}$ given by:\\
Multiset1=$-z_{12}=c$,\\
Multiset2=$z_{12}=d$\\
s.t their ratio $=\frac{c}{d}=-1$.\\
So,$a_2^c$ clearly holds. Hence, we are done, i.e., every homogeneous component of $A=\begin{bmatrix}
	z_{22}&1\\
	z_{12}&z_{13}\\
\end{bmatrix}$ belongs to $I_{v',w'}$.\\
Next choose, $A$ to be $\begin{bmatrix}
	z_{21}&1\\
	z_{11}&z_{13}\\
\end{bmatrix}$. Here, $M= \begin{bmatrix}
z_{11}\\
\end{bmatrix}$.\\
We can work out similarly for this matrix $A=\begin{bmatrix}
	z_{21}&1\\
	z_{11}&z_{13}\\
\end{bmatrix}$ and see that every homogeneous component of $A=\begin{bmatrix}
	z_{21}&1\\
	z_{11}&z_{13}\\
\end{bmatrix}$ belongs to $I_{v',w'}$.\\
Hence $I_{v',w'}$ is homogeneous.
\end{example} 
We can now state the following theorem (Theorem \ref{t.main-work}), whose proof is obvious:
\begin{theorem}\label{t.main-work}
Let $(v',w')\in S_n\times S_n$.\\
If for all inhomogeneous defining essential minors $A$ of $Z^{(v')}$, we have that the corresponding subminor $M$ of $A$ is non singular and the conditions (b)$^c$ and (a1 and a2)$^c$ of Proposition \ref{p.statement2andconverse} above hold, then the KL-ideal $I_{v',w'}$ is standard homogeneous. \\
Conversely, under the assumption (***) on $v',w'$, if a KL-ideal $I_{v',w'}$ is standard homogeneous, then given any inhomogeneous defining essential minor $A$ of $Z^{(v')}$ whose corresponding subminor $M$ is non singular, we have that the conditions (b)$^c$ and (a1 and a2)$^c$ of Proposition \ref{p.statement2andconverse} above hold.
\end{theorem}

\subsection{Relating to Neye's work}\label{ss.related-to-neye}
Recalling Proposition \ref{p.Neye}, suppose that $w'\leq v'$ in the tilde matrix order and that $\exists $ some pair $(a'_1a'_2a'_3 , c'_1c'_2c'_3)$ where $a'_1a'_2a'_3$  is a $123$ pattern in $v'$ and $c'_1c'_2c'_3$ is a $312$ pattern in $w'$ for which neither $c'_3 > a'_2$ nor $w'^{-1}(c'_2) < v'^{-1}(a'_2)$. That is, we are now in the complement of Neye's set. Hence forth, assume that we always have such a pair $(a'_1a'_2a'_3 , c'_1c'_2c'_3)$ where $w'\leq v'$ (in the tilde matrix order).
\begin{lemma}\label{l.existenceofminor} Given any pair  $(a'_1a'_2a'_3 , c'_1c'_2c'_3)$ as mentioned above,\\
(a) there always exists a point of $E_{w'}$ corresponding to it, call it $(x_0,y_0)$.\\
and\\
(b) For any inhomogeneous defining essential minor $A_{x_0y_0}$ corresponding to the point $(x_0,y_0)$ of $E_{w'}$, we have that $det(M_{x_0y_0})$ equals the lowest degree homogeneous component of $det(A_{x_0y_0})$, where $M_{x_0y_0}$ is the subminor of $A_{x_0y_0}$ obtained from it by removing all the rows and columns in it corresponding to the entry $1$.
\end{lemma}
\begin{proof}
	Proof of (a): Consider the box diagram to find out the essential set defined in (insert section) and notice there is no line passing through the point $(c'_3,c'_2)$. Because any horizontal line through it $\implies \exists i< w'^{-1}(c'_3) $ such that $w'(i)=c'_3$ which is absurd, and any vertical line passing through it $\implies c'_2>c'_3$ which violates the $312$ pattern of $c'_1c'_2c'_3$. Now, consider the minimum of the set \[
	A=\{x|(x,c'_2)\in A_{w'}\}.
	\]
	Clearly this set is non empty as $c'_3$ lies in it. let $x_0$ be the minimum of this set. Then, consider the maximum of the set:\[
	B=\{y|(x_0,y)\in A_{w'}\}.
	\]
	 which is again non empty as $c'_2$ lies in it. Let, $y_0$ be the maximum of this set, then $(x_0,y_0) \in E_{w'}$.\\
	  Because $(x_0,y_0+1)\in E_{w'}$ contradicts the maximality of $B$ and $(x_0-1,y_0) \in E_{w'}$ is impossible as $(x_0-1,c'_2)\notin A\implies $ either $x_0$ is the first row(which proves the point) or $\exists$ a horizontal line passing through $(x_0-1,c'_2)$ which then by definition passes through all $(x_0-1,y) \textnormal{ for } y \geq c'_2$.\\
	  Hence part (a) is proved. \\
Proof of (b): Let $A_{x_0y_0}$ be a $p$ minor with rows and columns $R=\{i_1,\dots i_p\}, S=\{j_1,\dots . j_p\}$ and those corresponding to 1s be $R'=\{r_1,\dots r_e\}, S'=\{s_1,\dots,s_e\} (e<p)$. So, the rows and columns of $M_{x_0y_0}$ are $R\setminus R'$ and $S\setminus S'$ respectively. Now, any non zero path $\hat{m}$ is of the form $\pm \Pi z_{ab}\cdot \Pi z_{cd}$ where $a\in R,c \in R\setminus R', b\in S,d\in S\setminus S'$. So whenever $z_{ab}\neq 1$ , we have $deg(\hat{m})\geq p-e$ and is exactly $p-e$ when $z_{ab}=1 \forall a\in R'$. Hence, the lowest degree homogeneous component is obtained by taking 1 from all the rows in $R$ and then all possible choice of variables from rows $R\setminus R'$ which is precisely of the form $1\cdot1\cdot \textnormal{a times}\cdot\cdot1\cdot det(M_{x_0y_0})$= $det(M_{x_0y_0})$.\\
	  This completes the proof of part (b). 
\end{proof}	
\begin{remark}\label{r.last-lemma}
With the notation of Lemma \ref{l.existenceofminor} above, it is easy to see that if $A_{x_0y_0}$ is a minor containing the rows and columns corresponding to all $1$'s lying inside the matrix $Z^{(v')}_{x_0y_0}$, then the subminor $M_{x_0y_0}$ of $A_{x_0y_0}$ has all its entries as variables.
\end{remark}
\begin{remark}\label{r.neye-complement}
For any pair $(v',w')$ in $S_n\times S_n$, which is in Neye's complement set, the first assertion of Theorem \ref{t.main-work} is applicable. 
\end{remark}
\begin{remark}\label{r.example-neye-complement}
Observe that in Example \ref{e.statement2andconverse} above, we can take $c_1'c_2'c_3'=645$ and $a_1'a_2'a_3'=156$. Then we are in Neye's complement set, and all the results that we have developed till now can be applied on this example for concluding that the KL-ideal $I_{v',w'}$ is standard homogeneous.
\end{remark}

\begin{acknowledgements}
	We would like to thank Professor Alexander Yong for his painstaking job of answering all our queries, when we were reading the paper \cite{woo2023schubert} to get ourselves introduced to the research problem presented in this paper. We would also like to thank Professor Sudhir R Ghorpade for his careful reading of the paper and valuable comments. \\
	This research is financially supported by the university grants commission (UGC) of India through the NET-JRF award, awarded to Mr. Adhip Ganguly.
\end{acknowledgements}








\end{document}